\documentclass[11pt]{amsart}
\usepackage[dvipdfmx]{graphicx, color}
\usepackage{amssymb, amsmath,amsthm}
\usepackage{epstopdf}

\newcommand{\nc}{\newcommand}
\nc{\dmo}{\DeclareMathOperator}
\nc{\nt}{\newtheorem}

\nt{theorem}{Theorem}
\nt{lemma}{Lemma}

\newtheorem{thm}{{\bf Theorem}}[section]
\newtheorem{lem}[thm]{{\bf Lemma}}

\newtheorem{rem}[thm]{Remark}
\newtheorem{ex}[thm]{Example}

\numberwithin{equation}{section}

\begin{document} 

\title[]{A construction  of  pseudo-Anosov braids  with small normalized entropies}

\author[S. Hirose]{%
Susumu Hirose
}
\address{%
Department of Mathematics,  
Faculty of Science and Technology, 
Tokyo University of Science, 
Noda, Chiba, 278-8510, Japan}
\email{%
hirose\_susumu@ma.noda.tus.ac.jp
}

\author[E. Kin]{%
   Eiko Kin
}
\address{%
       Department of Mathematics, Graduate School of Science, Osaka University Toyonaka, Osaka 560-0043, JAPAN
}
\email{%
        kin@math.sci.osaka-u.ac.jp
}



\date{%
March 12, 2020.
}


	
\begin{abstract} 
Let 
$b$ be a pseudo-Anosov braid whose permutation has a fixed point and 
let $M_b$ be the mapping torus by the pseudo-Anosov homeomorphism defined on the genus $0$ fiber $F_b$ 
associated with $b$. 
We prove that there is a $2$-dimensional  subcone $\mathcal{C}_0$ contained in the 
fibered cone $ \mathcal{C}$ of $F_b$ 
such that the fiber $F_a$ for each primitive integral class $a \in \mathcal{C}_0$ has genus $0$. 
We also give a constructive description of the monodromy $ \phi_a: F_a \rightarrow F_a$ of the fibration on $M_b$ over the circle, 
and consequently  provide  a construction of many sequences of pseudo-Anosov braids with small normalized entropies.  
As an application 
we prove that the smallest entropy among skew-palindromic braids with $n$ strands is  comparable to $1/n$, 
and the smallest entropy among elements of the odd/even spin mapping class groups of genus $g$ 
is comparable to $1/g$.

\end{abstract}
\maketitle

\section{Introduction} 

\begin{center}
\begin{figure}
\includegraphics[width=3in]{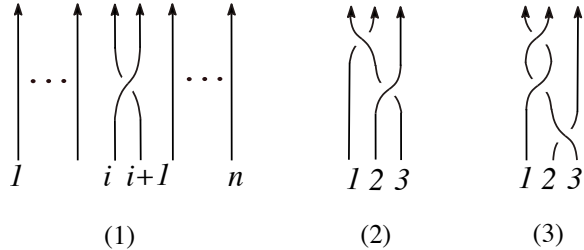}
\caption{(1) $\sigma_i $.  
(2) $\sigma_1^{-1} \sigma_2 $ with the permutation 
$1 \mapsto 2$, $2 \mapsto 3$, $3 \mapsto 1$. 
(3) $\sigma_1^2 \sigma_2^{-1} $ whose permutation has a fixed point.}
\label{fig_sigma_i}
\end{figure}
\end{center}

\begin{center}
\begin{figure}
\includegraphics[width=3.7in]{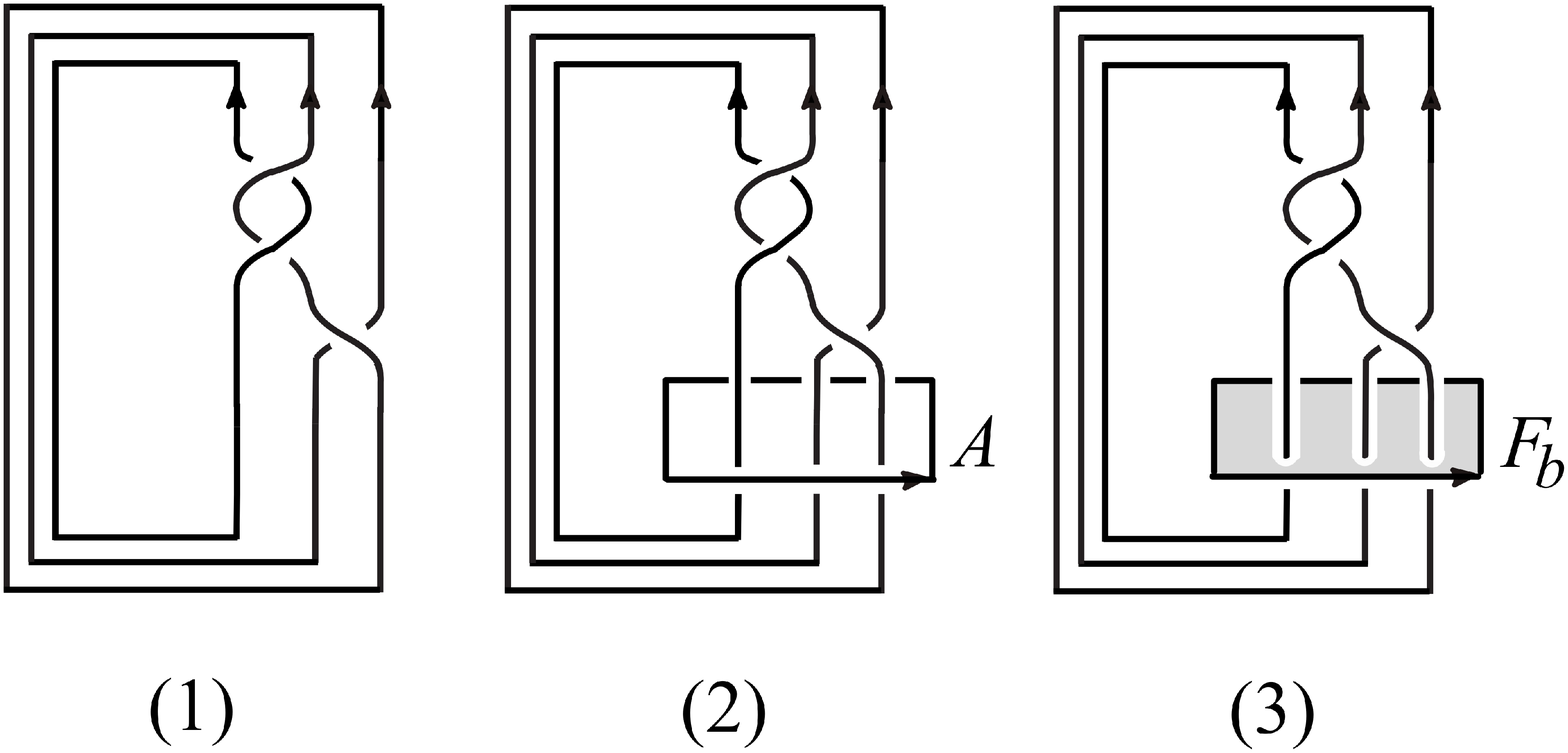}
\caption{
$b:= \sigma_1^2 \sigma_2^{-1}$. 
(1) $\mathrm{cl}(b)$. 
(2) $\mathrm{br}(b)$. 
(3) $F_b \hookrightarrow M_b$.}
\label{fig_closure}
\end{figure}
\end{center}

Let $\Sigma= \Sigma_{g,n}$ be an orientable surface of genus $g$ with $n $ punctures for $n\ge 0$.  
We set $\Sigma_g= \Sigma_{g,0}$. 
By mapping class group $\mathrm{Mod}(\Sigma_{g,n})$, 
we mean the group of isotopy classes of orientation preserving self-homeomorphisms 
on $\Sigma_{g,n}$ preserving  punctures setwise. 
By Nielsen-Thurston classification, 
elements  in $\mathrm{Mod}(\Sigma)$ are classified into three types: 
periodic, reducible, pseudo-Anosov \cite{Thurston88,FM12}. 
For $\phi  \in \mathrm{Mod}(\Sigma)$ 
we choose a representative $\Phi \in \phi$ and 
consider the mapping torus 
$M_{\phi}= \Sigma \times {\Bbb R}/ \sim$,   
where $\sim$ identifies $(x,t+1)$ with $ (\Phi(x),t)$ 
for $x \in \Sigma$ and $t \in {\Bbb R}$. 
Then $\Sigma$ is a fiber of a fibration on $M_{\phi}$ over the circle $S^1$ and 
$\phi$ is called the {\it monodromy}. 
A theorem by Thurston \cite{Thurston98} 
asserts that 
$M_{\phi}$ admits a hyperbolic structure of finite volume 
if and only if $\phi$ is pseudo-Anosov.

For a pseudo-Anosov element $\phi \in \mathrm{Mod}(\Sigma)$ 
there is a representative $\Phi: \Sigma \rightarrow \Sigma$ of $\phi$ 
called a {\it pseudo-Anosov homeomorphism} with the following property: 
$\Phi$ admits a pair of transverse measured foliations 
$(\mathcal{F}^u, \mu^u)$ and 
$(\mathcal{F}^s, \mu^s)$ 
and  a constant $\lambda = \lambda(\phi)  >1$ depending on $\phi$ such that 
$\mathcal{F}^u$ and $\mathcal{F}^s$ are invariant under $\Phi$, and 
$\mu^u$ and $\mu^s$ are uniformly multiplied by $\lambda$ and $\lambda^{-1} $ under $\Phi$. 
The constant $\lambda(\phi)$ is called the {\it dilatation} 
and $\mathcal{F}^u$ and $\mathcal{F}^s$ are called the {\it unstable} and {\it stable foliation}. 
We call the logarithm $\log(\lambda(\phi))$  the {\it entropy}, and call 
$$\mathrm{Ent}(\phi)= |\chi(\Sigma)| \log(\lambda(\phi))$$ 
the {\it normalized entropy} of $\phi$, 
where $\chi(\Sigma)$ is the Euler characteristic of $\Sigma$. 
Such normalization of the entropy is suited for the context of $3$-manifolds~\cite{FLM11,KM18}.

Penner~\cite{Penner91} proved that  
if $\phi \in  \mathrm{Mod}(\Sigma_{g,n})$ is pseudo-Anosov, then 
\begin{equation}
\label{equation_Penner}
\frac{\log 2}{12g- 12+ 4n} \le \log(\lambda(\phi)).
\end{equation}
See also \cite[Corollary 2]{KM18}. 
For a fixed surface $\Sigma$, 
the set 
$$\{\log \lambda(\phi)\ |\ \phi \in \mathrm{Mod}(\Sigma)\  \mbox{is pseudo-Anosov}\}$$ 
is a closed, discrete subset of ${\Bbb R}$ (\cite{AY81}).   
For any subgroup or subset $G \subset \mathrm{Mod}(\Sigma)$ 
let $\delta(G)$ denote the minimum of $\lambda(\phi)$ over all pseudo-Anosov elements $\phi \in G$. 
Then $\delta(G)  \ge \delta(\mathrm{Mod}(\Sigma))$. 
 We write $f \asymp h$ 
 if there is a universal constant $P>0$  such that 
 $1/P \le f/h \le P$. 
 It is proved by Penner~\cite{Penner91} that 
 the minimal entropy among pseudo-Anosov elements in $\mathrm{Mod}(\Sigma_g)$ on the closed surface of genus $g$ satisfies 
 $$\log \delta(\mathrm{Mod}(\Sigma_g)) \asymp \dfrac{1}{g}.$$
 See also \cite{HK06,Tsai09,Valdivia12} for other sequences of mapping class groups. 
 
 For any $P>0$, 
 consider the set $\Psi_P$ consisting of all pseudo-Anosov homeomorphisms $\Phi: \Sigma \rightarrow \Sigma$ 
 defined on any surface $\Sigma$ 
 with the normalized entropy $|\chi(\Sigma)| \log \lambda(\Phi) \le P$. 
 This is an infinite set in general (take $P > 2 \log (2+ \sqrt{3})$ for example) and 
 is well-understood in the context of hyperbolic fibered $3$-manifolds. 
 The universal finiteness theorem by Farb-Leininger-Margalit~\cite{FLM11} states that 
 the set of homeomorphism classes of mapping tori of pseudo-Anosov homeomprhisms 
 $\Phi^{\circ}: \Sigma^{\circ} \rightarrow \Sigma^{\circ}$ is finite, 
 where $\Phi^{\circ}: \Sigma^{\circ} \rightarrow \Sigma^{\circ}$ is the fully punctured pseudo-Anosov homeomprhism 
 obtained from $\Phi \in \Psi_P$. (Clearly $\lambda(\Phi^{\circ}) = \lambda(\Phi)$.)  
 In other words 
 such $\Phi^{\circ}: \Sigma^{\circ} \rightarrow \Sigma^{\circ}$ is a monodromy of a fiber 
  in some fibered cone 
 for a hyperbolic fibered $3$-manifold in the finite list determined by $P$.   
 Thus $3$-manifolds in the finite list govern all pseudo-Anosov elements in $\Psi_P$. 
 It is natural to ask the dynamics and  a constructive description of  elements in $ \Psi_P$. 
 There are some results about this question by several authors~\cite{Dehornoy13,Hironaka14,Kin15,KT11,Valdivia12}, 
 but it is not completely understood. 
 In this paper we restrict our attention to the pseudo-Anosov elements  in $ \Psi_P$ defined on the genus $0$ surfaces, 
 and provide an approach for a concrete description of those elements.

Let $B_n$ be the braid group with $n$ strands. 
The group $B_n$ is generated by the braids $\sigma_1, \cdots, \sigma_{n-1}$ as  in Figure~\ref{fig_sigma_i}. 
Let $\mathcal{S}_n$ be the symmetric group, the group of bijections of $\{1, \ldots, n\}$ to itself. 
A permutation $\mathcal{P} \in \mathcal{S}_n$ has a {\it fixed point} if $\mathcal{P}(i)= i$ for some $i$. 
We have a surjective homomorphism 
$\pi: B_n \rightarrow \mathcal{S}_n$ 
which sends each $\sigma_j$ to the transposition $(j,j+1)$.

The closure $\mathrm{cl}(b)$ of a braid $b \in B_n$ 
 is a knot or link in the $3$-sphere $S^3$. 
The {\it braided link} 
$$\mathrm{br}(b) = \mathrm{cl}(b) \cup A$$  
is a link in $S^3$ obtained from $\mathrm{cl}(b)$ 
with its braid axis $A$ 
(Figure~\ref{fig_closure}). 
Let $M_b$ denote the exterior of $\mathrm{br}(b)$ which is a $3$-manifold with boundary. 
It is easy to find an $(n+1)$-holed sphere $F_b$ in $M_b$ 
(Figure~\ref{fig_closure}(3)).  
Clearly $F_b$ is a fiber of a fibration on $M_b \rightarrow S^1$ 
and its monodromy $\phi_b: F_b \rightarrow F_b$ is determined  by  $b$. 
We call $F_b$ the {\it $F$-surface} for $b$.

A braid $b \in B_n$ is {\it periodic} (resp. {\it reducible}, {\it pseudo-Anosov}) 
if the associated mapping class $f_b \in \mathrm{Mod}(\Sigma_{0,n+1})$ is 
of the corresponding  type (Section~\ref{subsection_frombraids}). 
If $b$ is pseudo-Anosov, 
then the {\it dilatation} $\lambda(b)$  is defined by $\lambda(f_b)$ and 
the {\it normalized entropy} $\mathrm{Ent}(b)$  is defined by $\mathrm{Ent}(f_b) $. 
The following theorem is due to Hironaka-Kin~\cite[Proposition~3.36]{HK06} 
together with the observation by Kin-Takasawa~\cite[Section~4.1]{KT11}.

\begin{thm}
\label{thm_motivation}
There is  a sequence of pseudo-Anosov braids $z_n  \in B_n$ 
such that 
$\mathrm{Ent}(z_n) \ne  2 \log (2+ \sqrt{3})$, 
$M_{z_n} \simeq M_{\sigma_1^2 \sigma_2^{-1}}$ for each $n \ge 3$ 
and $\mathrm{Ent}(z_n) \to 2 \log (2+ \sqrt{3})$ as $n \to \infty$.  
\end{thm}

\noindent
Here $\simeq$ means they are homeomorphic to each other. 
The limit point $2 \log (2+ \sqrt{3})$ is equal to $\mathrm{Ent}(\sigma_1^2 \sigma_2^{-1})$. 
By the lower bound (\ref{equation_Penner}), 
Theorem~\ref{thm_motivation} implies that 
$$\log \delta(\mathrm{Mod}(\Sigma_{0,n})) \asymp \dfrac{1}{n}.$$ 
In particular, the hyperbolic fibered $3$-manifold 
$M_{\sigma_1^2 \sigma_2^{-1}}$ admits 
 an infinitely family of genus $0$ fibers of fibrations over $S^1$.

Let $z_n $ be a pseudo-Anosov braid with $d_n$ strands. 
We say that a sequence $\{z_n\}$ {\it has a small normalized entropy} 
if   $d_n \asymp n$ and 
there is a constant $P>0$ which does not depend on $n$ such that 
 $\mathrm{Ent}(z_n) \le P$. 
 By (\ref{equation_Penner}) 
 a sequence $\{z_n\}$ having a small normalized entropy  means 
 $\log (\lambda(z_n)) \asymp 1/n$.  
One of the aims in this paper is to give a construction of many sequences of pseudo-Anosov braids with small normalized entropies. 
The following result generalizes Thereom~\ref{thm_motivation}. 

\begin{theorem}
\label{thm_nomalizedentropy}
Suppose that $b$ is  a pseudo-Anosov braid whose  permutation has a fixed point. 
There is a sequence of  pseudo-Anosov braids $\{z_n\}$ with small normalized entropy 
such that 
$\mathrm{Ent}(z_n) \to \mathrm{Ent}(b)$ 
as $n \to \infty$ and 
$M_{z_n} \simeq M_b$ for $n \ge 1$. 
\end{theorem}

The  proof of Theorem~\ref{thm_nomalizedentropy} is constructive. 
In fact one can describe  braids $z_n$  explicitly. 
For a more general result see Theorems~\ref{thm_compact}, \ref{thm_sequence}. 
Let $\mathcal{C} \subset H_2(M_b, \partial M_b)$ be the fibered cone containing  $[F_b]$. 
A theorem by Thurston~\cite{Thurston86}  states that 
for each primitive integral class $a \in \mathcal{C}$ 
there is a connected fiber $F_a$ 
with the pseudo-Anosov monodromy $\phi_a: F_a \rightarrow F_a$ of a fibration on the hyperbolic $3$-manifold $M_b$ over $S^1$. 
The following theorem states a structure of $\mathcal{C}$.

\begin{theorem}
\label{thm_main}
Suppose that $b$ is  a pseudo-Anosov braid whose  permutation has a fixed point. 
Then there are a $2$-dimensional  subcone $\mathcal{C}_0 \subset \mathcal{C}$ 
and an  integer $u  \ge 1$ 
with the following properties. 
\begin{enumerate}
\item[(1)] 
The fiber $F_a$ for each primitive integral class $a \in \mathcal{C}_0$ has genus $0$. 

\item[(2)] 
The monodromy 
$ \phi_a: F_a \rightarrow F_a$ for each primitive integral class $a \in \mathcal{C}_0$   is conjugate to 
$$(\omega_1 \psi)  \cdots (\omega_{u-1} \psi) (\omega_u \psi) \psi^{m-1}: F_a \rightarrow F_a,$$  
where $m \ge 1$ depends on the class $a$,  
$\psi$ is periodic and 
each $\omega_j$ is reducible. 
Moreover there are homeomorphisms $\widehat{\omega}_j: S_0 \rightarrow S_0$ on a surface $S_0$  for $j = 1, \ldots,  u$ 
determined by $b$ 
and an embedding $h: S_0 \hookrightarrow F_a$
such that 
$h(S_0)$ is the support of each $w_j$ 
and 
$$w_j|_{h(S_0)} =  h \circ \widehat{\omega}_j \circ h^{-1}.$$
\end{enumerate}
\end{theorem}

Theorem~\ref{thm_main}  gives a constructive description of  $\phi_a$. 
Also it states that 
each $w_j: F_a \rightarrow F_a$ is reducible supported on a uniformly bounded subsurface $h(S_0) \subset F_a$. 
It turns out from the proof that 
the type of the periodic homeomorphism $\psi: F_a \rightarrow F_a$  does not depend on $a \in C_0$
(Remark~\ref{rem_periodic}), 
see Figure~\ref{fig_flm}(1). 
Theorem~\ref{thm_main} reminds us of the symmetry conjecture in \cite{Margalit18} by 
Farb-Leininger-Margalit.

Clearly the permutation of each pure braid has a fixed point. 
For any pseudo-Anosov braid $b$,  a suitable power $b^k$ becomes a pure braid  
and one can apply Theorems~\ref{thm_nomalizedentropy}, \ref{thm_main} for $b^k$.

We have a remark about Theorem~\ref{thm_nomalizedentropy}. 
Theorem~10.2 in \cite{McMullen00} by McMullen also tells us 
the existence of a sequence $(F_n, \phi_n)$ of fibers and monodromies in $\mathcal{C}$ 
such that 
$\mathrm{Ent}(\phi_n) \to \mathrm{Ent}(b)$ as $n \to \infty$ and $|\chi(F_n)| \asymp n$. 
However one can not appeal his theorem for the genera of fibers $F_n$. 
Theorem~\ref{thm_nomalizedentropy} says that $F_n$ has genus $0$ in fact.

\begin{center}
\begin{figure}
\includegraphics[width=3.5in]{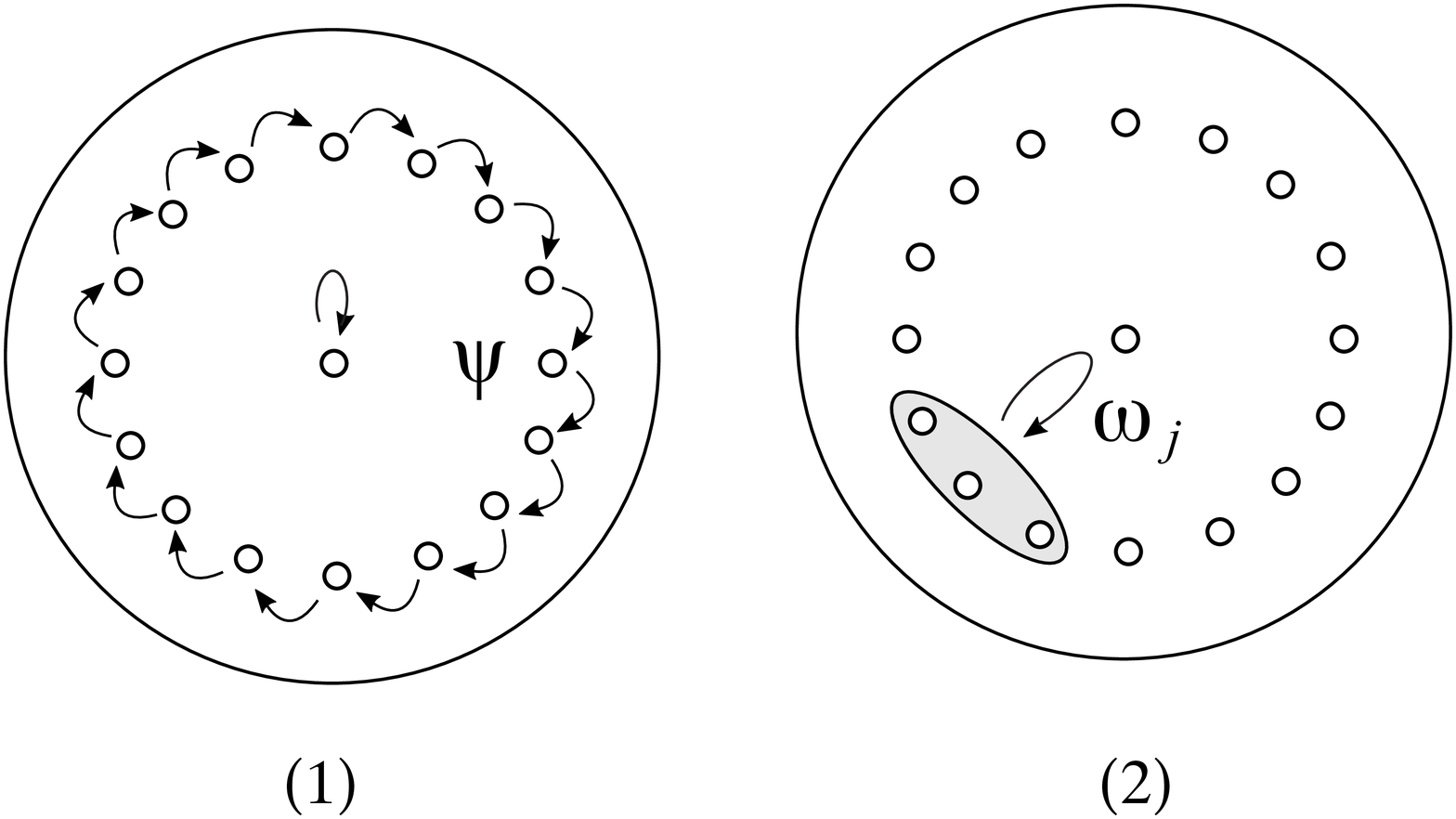}
\caption{
Dynamics of $\psi$ and $\omega_j$ in Theorem~\ref{thm_main}. 
(1) Periodic $ \psi: F_a \rightarrow F_a$. 
(2) Reducible $\omega_j: F_a \rightarrow F_a$. 
Subsurface 
$h(S_0)$ is shaded.}
\label{fig_flm}
\end{figure}
\end{center}

\begin{center}
\begin{figure}
\includegraphics[width=2.5in]{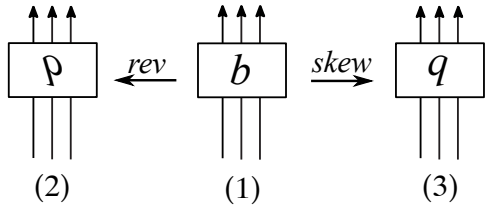}
\caption{
Illustration of braids 
(1) $b$, 
(2) $rev(b)$, 
(3) $skew(b)$.}
\label{fig_palyobi}
\end{figure}
\end{center}


As an application 
we will determine asymptotic behaviors of the minimal dilatations 
of a subset of $B_n$ consisting of braids with a symmetry.   
A braid $b \in B_n$ is {\it palindromic} if $rev(b)= b$, 
where 
$rev: B_n \rightarrow B_n$ is a map such that 
if $w$ is a word of letters $\sigma_j^{\pm 1}$ representing $b$, 
then $rev(b)$  is 
the braid obtained from $b$ reversing the order of letters in $w$. 
A braid $b \in B_n$ is {\it skew-palindromic} if  $skew(b)= b$, 
where $skew(b)= \Delta rev(b) \Delta^{-1}$ and 
$\Delta$ is a half twist (Section~\ref{subsection_braidgroups}). 
See Figure~\ref{fig_palyobi}. 
We will prove that 
 dilatations of palindromic braids have the following  lower bound.

\begin{theorem}
\label{thm_pal}
If  $b \in B_n$ is palindromic and pseudo-Anosov for $n \ge 3$, 
then 
$$\lambda(b) \ge \sqrt{2+ \sqrt{5}}.$$
\end{theorem}
In contrast with palindromic braids 
we have the following result. 

\begin{theorem}
\label{thm_skp}
Let $PA_n$ be the set of skew-palindromic elements in $B_n$. 
We have 
$$\log \delta(PA_n) \asymp \dfrac{1}{n}. $$ 
\end{theorem}

\begin{center}
\begin{figure}
\includegraphics[width=5.2in]{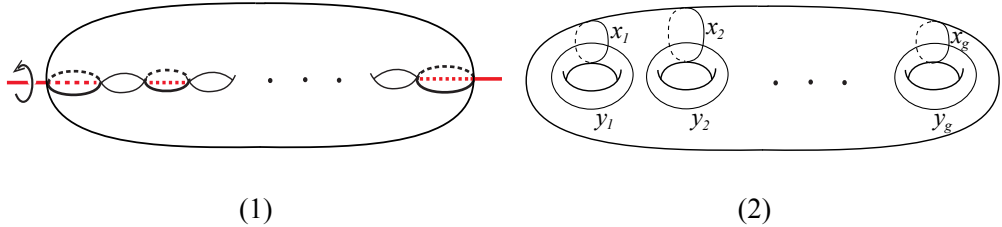} 
\caption{(1) $\mathcal{I}: \Sigma_g \rightarrow \Sigma_g$. 
(2) A basis $\{ x_1, y_1, \ldots, x_g, y_g\}$ of $H_1 (\Sigma_g; \mathbb{Z}_2)$.}
\label{fig_basis}
\end{figure}
\end{center}

The {\it hyperelliptic mapping class group} $\mathcal{H}(\Sigma_g)$ 
is the subgroup of $\mathrm{Mod}(\Sigma_g)$ consisting of elements with representative homeomorphisms
that commute with some fixed hyperelliptic involution $\mathcal{I}: \Sigma_g \rightarrow \Sigma_g$ 
as  in Figure~\ref{fig_basis}(1). 
It is shown in \cite{HK06} that $\log \delta(\mathcal{H}(\Sigma_g)) \asymp 1/g$. 
See also \cite{FLM08,Hironaka14,HK17} for other subgroups of $\mathrm{Mod}(\Sigma_g)$. 
As an application we will 
determine the asymptotic behavior of the minimal dilatations of the odd/even spin mapping class groups of genus $g$. 
To define these subgroups  
let $(\cdot, \cdot)_2$ be the mod-$2$ intersection form on $H_1 (\Sigma_g; \mathbb{Z}_2)$.  
A map $\mathfrak{q}: H_1(\Sigma_g; \mathbb{Z}_2) \to \mathbb{Z}_2$ is a {\em quadratic form\/} if 
$\mathfrak{q}(v+w) = \mathfrak{q}(v) + \mathfrak{q}(w) + (v,w)_2$  for  
$v,w \in H_1(\Sigma_g; \mathbb{Z}_2)$. 
For a quadratic form $\mathfrak{q}$, 
the {\em spin mapping class group\/} $\mathrm{Mod}_g[\mathfrak{q}]$ 
is the subgroup of $\mathrm{Mod}(\Sigma_g)$ consisting of elements $\phi $ 
such that $\mathfrak{q} \circ \phi_* = \mathfrak{q}$. 
To define the two quadratic forms $\mathfrak{q}_0$ and $\mathfrak{q}_1$ 
we choose a basis $\{x_1, y_1, \ldots, x_g, y_g \}$ of 
$H_1 (\Sigma_g; \mathbb{Z}_2)$ as  in Figure~\ref{fig_basis}(2). 
Let $\mathfrak{q}_0$ be the quadratic form such that 
$\mathfrak{q}_0(x_i)=\mathfrak{q}_0(y_i)=0$ for $1 \le i \le g$. 
 Let $\mathfrak{q}_1$ be the quadratic form such that 
$\mathfrak{q}_1(x_1)=\mathfrak{q}_1(y_1)=1$  and 
$\mathfrak{q}_1(x_i)=\mathfrak{q}_1(y_i)=0$ for  $ 2 \le i \le g$. 
A result of Dye~\cite{Dye78} tells us that  
$\mathrm{Mod}_g[\mathfrak{q}]$ for any $\mathfrak{q}$ is conjugate to either 
$\mathrm{Mod}_g[\mathfrak{q}_0]$ or $\mathrm{Mod}_g[\mathfrak{q}_1]$ 
in $\mathrm{Mod}(\Sigma_g)$. 
We call $\mathrm{Mod}_g[\mathfrak{q}_0]$ and $\mathrm{Mod}_g[\mathfrak{q}_1]$ 
the {\it even spin} and  {\it odd spin mapping class group} respectively. 
It is known that $\mathrm{Mod}_g[\mathfrak{q}_1]$ 
attains the minimum index for a proper subgroup of $\mathrm{Mod}(\Sigma_g)$ 
and $\mathrm{Mod}_g[\mathfrak{q}_0]$ attains the secondary minimum, 
see Berrick-Gebhardt-Paris~\cite{BGP14}.

\begin{theorem}
\label{thm_spin}
We have 
\begin{enumerate}
\item[(1)] 
$\log \delta (\mathrm{Mod}_g[\mathfrak{q}_1] \cap \mathcal{H}(\Sigma_g)) \asymp \dfrac{1}{g}$ and 
\smallskip

\item[(2)] 
$\log \delta (\mathrm{Mod}_g[\mathfrak{q}_0] \cap \mathcal{H}(\Sigma_g)) \asymp \dfrac{1}{g}$. 
\end{enumerate}
In particular 
$\log \delta (\mathrm{Mod}_g[\mathfrak{q}]) \asymp 1/g$ for each quadratic form $\mathfrak{q}$. 
\end{theorem}

\noindent
{\bf Acknowledgments.} 
We would like to thank Mitsuhiko Takasawa for helpful conversations and comments. 
The first author was supported by 
Grant-in-Aid for
Scientific Research (C) (No. 16K05156),
Japan Society for the Promotion of Science. 
The second author was supported by 
Grant-in-Aid for
Scientific Research (C) (No. 18K03299), 
Japan Society for the Promotion of Science.


\section{Preliminaries} 
\label{section_preliminaries}

\subsection{Links}

Let $L$ be a link in the $3$-sphere $S^3$. 
Let $\mathcal{N}(L)$ denote a  tubular neighborhood  of $L$ 
and  let $\mathcal{E}(L)$ denote the exterior of $L$, i.e. $\mathcal{E}(L)=S^3 \setminus \mathrm{int}(\mathcal{N}(L))$.

Oriented links $L$ and $L'$ in $S^3$ are {\it equivalent}, 
denoted by $L \sim L'$ 
if there is an orientation preserving homeomorphism $f: S^3 \rightarrow S^3$  
such that 
$f(L)= L'$ with respect to the orientations of the links. 
Furthermore for components $K_i$ of $L$ and $K_i'$ of $L'$ with $i= 1, \ldots, m$ 
if $f$ satisfies $f(K_i)= K_i'$ for each $i$, 
then 
$(L, K_1, \ldots, K_m)$ and $(L', K'_1, \ldots, K'_m)$ are {\it equivalent} 
and we write 
$$(L, K_1, \ldots, K_m) \sim (L', K'_1, \ldots, K'_m).$$

\subsection{Braid groups $B_n$ and spherical braid groups $SB_n$}  
\label{subsection_braidgroups}

%
%


%

Let 
$\delta_j= \sigma_1 \sigma_2 \cdots \sigma_{j-1}$ and  
 $\rho_j = \sigma_1 \sigma_2 \cdots \sigma_{j-2} \sigma_{j-1}^2$. 
 The half twist $\Delta_j$ is given by 
 $\Delta_j= \delta_j \delta_{j-1} \cdots \delta_2$.  
 We often omit the subscript $n$ in $\Delta_n$, $\delta_n$ and $\rho_n$ 
when they are precisely $n$-braids.

We put  indices $1, 2, \ldots, n$ from left to right on the bottoms of strands, 
and give an orientation of strands from the bottom to the top (Figure~\ref{fig_sigma_i}). 
The closure $\mathrm{cl}(b)$ is oriented by the strands. 
We think of 
$\mathrm{br}(b) = \mathrm{cl}(b) \cup A$   
as an oriented link in $S^3$ 
choosing   an orientation of $A= A_b$ arbitrarily. 
(In Section~\ref{section_imonotonic} we assign an orientation of the braid axis for {\it $i$-monotonic braids}).

If two braids are conjugate to each other, then 
their braided links are equivalent. 
Morton proved that the converse holds if their axises are preserved. 

\begin{thm}[Morton~\cite{Morton78}]
\label{thm_Morton}
If $(\mathrm{br}(b), A_b)$ is equivalent to $(\mathrm{br}(c), A_c)$ for braids $b, c \in B_n$, 
then $b$ and $c$ are conjugate in $B_n$.  
\end{thm}

Let us turn to the spherical braid group $SB_n$ with $n$ strands. 
We also denote by $\sigma_i$, 
the element of $SB_n$ as shown in Figure~\ref{fig_sigma_i}(1). 
The group $SB_n$ is  generated by $\sigma_1, \ldots, \sigma_{n-1} $. 
For a braid $b \in B_n$ represented by a word of letters $\sigma_j^{\pm 1}$,  
let $S(b)$ denote the element in $SB_n$ represented by the same word as $b$.

For a braid $b$ in  $B_n$ or $SB_n$ 
 the {\it degree} of $b$ means the number $n$ of the strands, 
 denoted by $d(b)$.

\subsection{Mapping classes and mapping tori from braids}
\label{subsection_frombraids}

Let $D_n$ be the $n$-punctured disk. 
Consider the mapping class group $\mathrm{Mod}(D_n)$, 
the group of isotopy classes of orientation preserving self-homeomorphisms on $D_n$ 
preserving  the boundary $\partial D$ of the disk setwise. 
We have a surjective homomorphism 
$$\Gamma: B_n \rightarrow \mathrm{Mod}(D_n)$$
which sends each generator $\sigma_i$ to the right-handed half twist $\mathfrak{t}_i$ 
between the $i$th and $(i+1)$st punctures. 
The kernel of $\Gamma$ is an infinite cyclic group generated by the full twist $\Delta^2 $.

Collapsing $\partial D$ to a puncture in the sphere 
we  have a homomorphism 
$$\mathfrak{c}: \mathrm{Mod}(D_n) \rightarrow \mathrm{Mod}(\Sigma_{0, n+1}).$$
We say that $b \in B_n$ is {\it periodic} (resp. {\it reducible}, {\it pseudo-Anosov}) if 
$f_b:= \mathfrak{c}(\Gamma(b)) $ is of the corresponding Nielsen-Thurston type. 
The braids $\delta, \rho \in B_n$ are periodic 
since some power of each braid is the full twist: 
$\Delta^2= \delta^n=\rho^{n-1} \in B_n$.

We also have a surjective homomorphism 
$$\widehat{\Gamma}: SB_n \rightarrow \mathrm{Mod}(\Sigma_{0,n})$$
sending each generator $\sigma_i$ to the right-handed half twist $\mathfrak{t}_i$. 
We say that $\eta \in SB_n$ is {\it pseudo-Anosov} if $ \widehat{\Gamma}(\eta) \in \mathrm{Mod}(\Sigma_{0,n})$ is  pseudo-Anosov. 
In this case $\lambda(\eta)$ is defined by the dilatation of $\widehat{\Gamma}(\eta)$.

 \subsection{Stable foliations $\mathcal{F}_b$ for pseudo-Anosov braids $b$}
 \label{subsection_stablefoliations}

 \begin{center}
\begin{figure}
\includegraphics[width=1.5in]{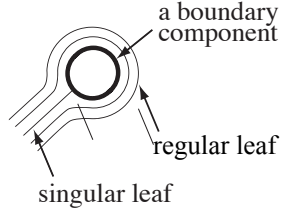}
\caption{Stable foliation which is $1$-pronged at a boundary component.}  
\label{fig_1prong}
\end{figure}
\end{center}

 Recall the surjective homomorphism $\pi: B_n \rightarrow \mathcal{S}_n$. 
We write  $\pi_b=  \pi(b)$ for $b \in B_n$. 
Consider a pseudo-Anosov braid $b \in B_n$ with $\pi_b(i)= i$. 
Removing the $i$th strand $b(i)$  from $b$, 
we get a braid $b-b(i) \in B_{n-1}$. 
Taking its spherical element, we have 
$S(b-b(i)) \in SB_{n-1}$. 
Note that $b-b(i) $ and $S(b-b(i)) $ are not necessarily pseudo-Anosov. 
A well-known criterion  uses the stable foliation $\mathcal{F}_b$ 
for the  monodromy $\phi_b: F_b \rightarrow F_b$ of a fibration on $M_b \rightarrow S^1$ 
as we recall now. 
Such a  fibration on $M_b$ extends naturally to a fibration on the manifold obtained from $M_b$ by Dehn filling a cusp 
along the boundary slope of the fiber $F_b$ which lies on the torus $\partial \mathcal{N}(\mathrm{cl}(b(i)))$.  
Also $\phi_b$ extends to the monodromy 
defined on $F_b^{\bullet} $ 
of the extended fibration, 
where  $F_b^{\bullet}$ is obtained from $F_b$ by filling in the boundary component of $F_b$ 
which lies on  $\partial \mathcal{N}(\mathrm{cl}(b(i)))$
with a disk. 
Then $b-b(i)$ is the corresponding braid for the extended monodromy defined on $F_b^{\bullet} $. 
Suppose that $\mathcal{F}_b$ is not $1$-pronged at the boundary component in question. 
(See Figure~\ref{fig_1prong} in the case where $F_b$ is $1$-pronged at a boundary component.)  
Then $\mathcal{F}_b$ extends  to the stable foliation  for 
$b-b(i)$, and hence $b-b(i)$  is pseudo-Anosov with the same dilatation as $b$. 
Furthermore if $\mathcal{F}_b$ is not $1$-pronged at the boundary component of $F_b$ 
which lies on $\partial \mathcal{N}(A)$, 
then  $S(b-b(i))$ is still pseudo-Anosov with the same dilatation as $b$.

 \subsection{Thurston norm}
 \label{subsection_thurstonnorm}

 Let $M$ be a $3$-manifold with boundary  (possibly $\partial M = \emptyset$). 
If $M$ is hyperbolic, i.e. the interior of $M$ possess a complete hyperbolic structure of finite volume, 
then there is a norm $\| \cdot \|$ on $H_2(M, \partial M; {\Bbb R})$, now called the Thurston norm~\cite{Thurston86}. 
The norm $\| \cdot \|$ has the property such that 
for any integral class $a  \in H_2(M, \partial M; {\Bbb R})$, 
$\|a\|= \min_S \{- \chi(S)\}$, 
where 
 the minimum is taken over all  oriented surface $S$ embedded in $M$ with $a= [S]$ 
and with no components  of non-negative Euler characteristic. 
The surface $S$ realizing this minimum  is called  a  {\it norm-minimizing surface} of $a $.

\begin{thm}[Thurston~\cite{Thurston86}]
\label{thm_norm}
The norm $\| \cdot \|$ on $H_2(M, \partial M; {\Bbb R})$ has the following properties. 
\begin{enumerate}
\item[(1)]
There are a set of maximal open cones 
$\mathcal{C}_1, \cdots, \mathcal{C}_k$ in $H_2(M, \partial M; {\Bbb R})$ 
and a bijection between the set of isotopy classes of connected fibers of fibrations $M \rightarrow S^1$ 
and the set of primitive integral classes in the union $\mathcal{C}_1 \cup \cdots \cup \mathcal{C}_k$. 

\item[(2)] 
The restriction of $\|\cdot\|$ to $\mathcal{C}_j$ is linear for each $j$. 

\item[(3)] 
If we let $F_a$ be a fiber of a fibration $M \rightarrow S^1$ associated with a primitive integral class $a$ in each $\mathcal{C}_j$, 
then $\|a\| = - \chi (F_a)$. 
\end{enumerate}
\end{thm}
We call the open cones $\mathcal{C}_j$  {\it fibered cones} 
and call integral classes in  $\mathcal{C}_j$  {\it fibered classes}.

\begin{thm}[Fried~\cite{Fried82}] 
\label{thm_Fried_1}
For a fibered cone $\mathcal{C}$ of a hyperbolic $3$-manifold $M$, 
there is a continuous function 
$\mathrm{ent}: \mathcal{C}  \rightarrow {\Bbb R}$ 
with the following properties. 
\begin{enumerate}

\item[(1)] 
For the monodromy  $\phi_a: F_a \rightarrow F_a$ of a fibration $M \rightarrow S^1$ 
associated with a primitive integral class $a \in \mathcal{C}$, 
we have 
$\mathrm{ent}(a) = \log (\lambda(\phi_a))$.

\item[(2)] 
$\mathrm{Ent}= \| \cdot\| \mathrm{ent}: \mathcal{C} \rightarrow  {\Bbb R}$  
is a continuous function which becomes constant on each ray through the origin.

\item[(3)]
If a sequence $\{a_n\}  \subset \mathcal{C}$ tends to a point $ \ne 0 $ 
in the boundary $\partial \mathcal{C}$
as $n $ tends to $\infty$, 
then $\mathrm{ent}(a_n) \to \infty$. 
In particular $\mathrm{Ent}(a_n) = \| a_n\| \mathrm{ent}(a_n)  \to \infty$. 
\end{enumerate}
\end{thm}

\noindent
We call $\mathrm{ent}(a)$ and $\mathrm{Ent}(a)$ the {\it entropy} and  {\it normalized entropy} of 
the class $a \in \mathcal{C}$.


For a pseudo-Anosov element $\phi \in \mathrm{Mod}(\Sigma)$ 
we consider the mapping torus $M_{\phi}$. 
The vector field 
$\frac{\partial}{ \partial t}$ on $\Sigma \times {\Bbb R}$ induces a  flow $\phi^t$ on $M_{\phi}$ 
called the {\it suspension flow}.

\begin{thm}[Fried~\cite{Fried79}]
\label{thm_Fried}
Let $\phi$ be a pseudo-Anosov mapping class defined on $\Sigma$ 
with stable and unstable foliations $\mathcal{F}^s$ and $\mathcal{F}^u$.  
Let $\widehat{\mathcal{F}^s}$ and $\widehat{\mathcal{F}^u}$ 
denote the suspensions of $\mathcal{F}^s$ and $\mathcal{F}^u$ by $\phi$. 
If $\mathcal{C}$ is a fibered cone  containing the fibered class $[\Sigma]$, 
then 
we can modify 
a norm-minimizing surface 
$F_a$ associated with each  primitive integral class $a \in \mathcal{C}$
by an isotopy on $M_{\phi}$ with  the following properties. 
\begin{enumerate}
\item[(1)] 
$F_a$ is transverse to the suspension flow $\phi^t$, 
and the first return map $\phi_a:F_a \rightarrow F_a$ is precisely the pseudo-Anosov monodromy 
of the fibration on $M_{\phi} \rightarrow S^1$ associated with $a$. 
Moreover $F_a$ is unique up to isotopy along flow lines. 

\item[(2)] 
The stable and unstable foliations for  $\phi_a$ are given by 
$\widehat{\mathcal{F}^s} \cap F_a$ and $\widehat{\mathcal{F}^u} \cap F_a$.  
\end{enumerate}
\end{thm}

\subsection{Disk twist}
\label{subsection_disktwist}

\begin{figure}
\includegraphics[width=2.8in]{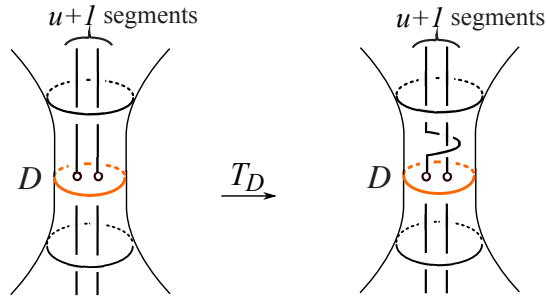}
\caption{Disk twist $T_D$.} 
\label{fig_dtwist}
\end{figure}

Let $L$ be a link in $S^3$. 
Suppose an unknot  $K$ is a component of $L$. 
Then the exterior $\mathcal{E}(K)$ (resp. $\partial \mathcal{E}(K)$)
is a solid torus (resp. torus). 
We take a disk $D$ bounded by the longitude of a tubular neighborhood  $ \mathcal{N}(K)$ of $K$. 
We define a mapping class $T_D$ defined on $\mathcal{E}(K)$ 
as follows. 
We cut $ \mathcal{E}(K)$ along $D$. 
 We have  resulting two sides obtained from $D$, and reglue two sides by twisting either of the sides  $360$ degrees 
so that the mapping class defined on $\partial  \mathcal{E}(K)$ 
is the right-handed Dehn twist about $\partial D$. 
Such a mapping class on $ \mathcal{E}(K)$ is called the {\it disk twist about} $D$. 
For simplicity we also call a self-homeomorphism representing the mapping class $T_D$ 
the {\it disk twist} about  $D$, 
and denote it by the same notation 
$$T_D:  \mathcal{E}(K) \rightarrow  \mathcal{E}(K) .$$
Clearly  $T_D$ equals the identity map outside a neighborhood of $D$ in $\mathcal{E}(K)$. 
We observe that if $u+1$ segments of $L-K$ pass through $D$ for $u \ge 1$, then 
$T_D(L- K)$ is obtained from $L- K$ by adding the full twist  near $D$. 
In the case  $u=1$, see Figure~\ref{fig_dtwist}. 
We may assume that $T_D$ fixes one of these segments, 
since any point in $D$ becomes the center of the twisting about $D$. 

For any integer $\ell $, consider a homeomorphism 
$$ T_D^{\ell}:  \mathcal{E}(K) \rightarrow   \mathcal{E}(K).$$
Observe that $ T_D^{\ell}$  converts $L $ into a link 
 $K \cup T_D^{\ell}(L- K)$ 
 such that $S^3 \setminus L$ is homeomorphic to $S^3 \setminus (K \cup T_D^{\ell}(L- K))$. 
 Then $ T_D^{\ell}$ induces a homeomorphism between the exteriors of links 
 \begin{equation}
 \label{equation_homeo}
 h_{D,\ell}: \mathcal{E}(L) \rightarrow \mathcal{E}(K \cup T_D^{\ell}(L-K)).
 \end{equation}
 We use the homeomorphism in (\ref{equation_homeo}) in later section.

\section{$i$-increasing braids and Theorem~\ref{thm_positivemain}}
\label{section_imonotonic}

\begin{center}
\begin{figure}
\includegraphics[width=2.8in]{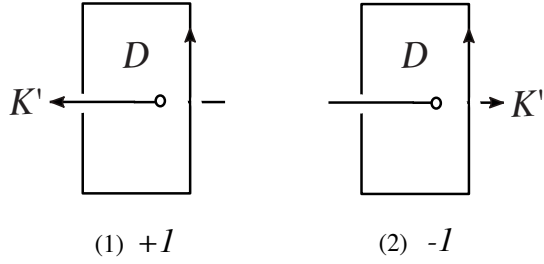} 
\caption{Sign of the point of intersection:  
$+1$ in  (1) and $-1$ in (2).}
\label{fig_intersection}
\end{figure}
\end{center}

\begin{center}
\begin{figure}
\includegraphics[width=4in]{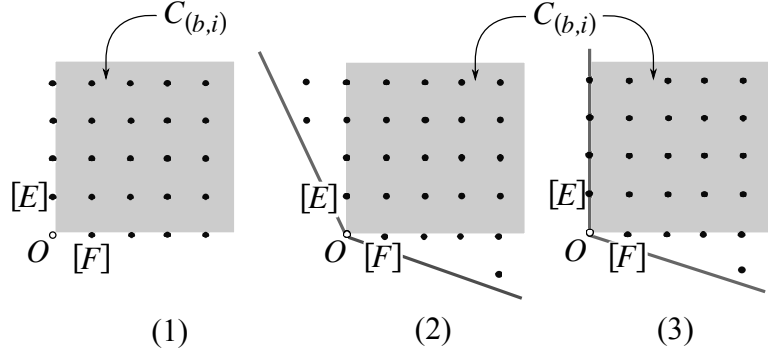} 
\caption{$F:= F_b$ and $E:= E_{(b,i)}$. 
(1) Subcone $ C_{(b,i)}$. 
(2)(3) Possible shapes of $\mathcal{C} \cap \{x[F]+ y[E]\ |\ x,y \in {\Bbb R}\}$. 
In case (2), 
$[E] \in \mathcal{C}$. 
In case (3), 
$[E] \notin \mathcal{C}$.} 
\label{fig_levelset}
\end{figure}
\end{center}

\subsubsection*{Definitions of $i$-increasing braids,  signs and intersection numbers}
Let $L$ be an oriented link in $S^3$ with a trivial component $K$. 
We take an oriented disk $D$ bounded by the longitude  of $\mathcal{N}(K)$ 
so that the orientation of $D$ agrees with the orientation of $K$. 
For each component $K'$ of $L-K$ such that 
$D$ and $K'$ intersect transversally with $D \cap K' \ne \emptyset$, 
we assign each point of intersection  $+1$ or $-1$ 
as shown in Figure~\ref{fig_intersection}.

Let $b $ be a braid with $\pi_b(i)= i$. 
We consider an oriented disk $D= D_ {(b,i)}$ bounded by the longitude $\ell_i$ of $\mathcal{N}(\mathrm{cl}(b(i)))$. 
Such a disk $D$ is unique up to isotopy on $\mathcal{E}(\mathrm{cl}(b(i)))$. 
We say that  a braid $b \in B_n$ with $\pi_b(i)= i$ is {\it $i$-increasing} (resp. {\it $i$-decreasing}) 
if 
there is a disk $D= D_{(b,i)} $ as above 
with the following conditions. 
\begin{enumerate}
\item[(D1)]
There is at least one component $K'$ of $\mathrm{cl}(b-b(i))$ 
such that $D \cap K' \ne \emptyset$. 

\item[(D2)]
Each component of $\mathrm{cl}(b-b(i))$ and $D$ intersect with each other transversally, 
and every point of intersection has the sign $+1$ (resp. $-1$). 
\end{enumerate}
We set $\epsilon(b,i)=1$ (resp. $\epsilon(b,i)= -1$), 
and call it the {\it sign} of the pair $(b,i)$. 
We also call $D$ the {\it associated disk} of the pair $(b,i)$. 
We say that $b$ is {\it $i$-monotonic} if $b$ is $i$-increasing or $i$-decreasing. 
Then we set 
$$I(b,i)= D \cap \mathrm{cl}(b-b(i))$$
and let $u(b,i) \ge 1$ be the cardinality of  $I(b,i)$. 
We call $u(b,i)$ the {\it intersection number} of the pair $(b,i)$. 
If the pair $(b,i)$ is specified, 
then we simply denote  $\epsilon(b,i)$ and  $u(b,i)$ by $\epsilon$ and  $u$ respectively.  
For example 
$ \sigma_1^{2} \sigma_2^{-1} $ is $1$-increasing with $u(\sigma_1^{2} \sigma_2^{-1},1)=1$. 

A braid $b $ is {\it positive} 
if $b$ is represented by a word in letters $\sigma_j$, but not $\sigma_j^{-1}$. 
A braid $b$ is irreducible if the Nielsen-Thurston type of $b$ is not reducible.

\begin{lem}
\label{lem_positive}
Let $b$ be a positive braid with  $\pi_b(i)= i$.
Then $b$ is $i$-increasing if $b$ is irreducible. 
\end{lem}

\begin{proof}
Suppose that a positive braid $b$ with  $\pi_b(i)= i$ is irreducible. 
Since $b$ is positive, there is a disk $D= D_{(b,i)}$ with the condition $(D2)$. 
Assume that $D$ fails in  $(D1)$. 
Let $\partial D_n$ be the boundary of the disk $D_n$ containing $n$ punctures. 
Consider a neighborhood of $\partial D_n \cup (D_n \cap D)$ in $D_n$ 
which is an annulus. 
One of the boundary components of this annulus is  an essential simple closed curve in $D_n$ 
preserved by $\Gamma(b) \in \mathrm{Mod}(D_n)$. 
This means that $b$ is reducible, a contradiction. 
Thus  $D$  satisfies $(D1)$, and $b$ is $i$-increasing. 
\end{proof}

\subsubsection*{Orientation of the axis $A$} 
Let $b $ be  $i$-monotonic with  $ \epsilon(b,i)= \epsilon$ and $u(b,i)= u$.  
Consider the braided link $\mathrm{br}(b)= \mathrm{cl}(b) \cup A$. 
The associated disk $D$ has a unique point of intersection with  $A$, 
and the cardinality of $I(b,i) \cup (D \cap A) $ is  $u(b,i)+1$. 
To deal with $\mathrm{br}(b)= \mathrm{cl}(b) \cup A$ as an oriented link, 
we consider an orientation of $\mathrm{cl}(b)$ as we described before, and 
assign an orientation of $A $ 
so that the sign of the intersection between  $D$ and $A$ coincides with $\epsilon(b,i)$. 
See Figure~\ref{fig_closure}(2). 

Recall that $M_b= \mathcal{E}(\mathrm{br}(b))$ is the exterior of $\mathrm{br}(b)$ which is a surface bundle over $S^1$. 
We consider an orientation of the $F$-surface $F_b$ which agrees with the orientation of $A$.

\subsubsection*{$E$-surface}  
We now define an oriented surface $E_{(b,i)}$ of genus $0$ embedded in $M_b$. 
Consider small $u(b,i)+1$ disks in the oriented disk $D=D_{(b,i)}$ whose centers are points of $I(b,i) \cup (D \cap A) $. 
Then  $E_{(b,i)} $ is a  sphere with $u(b,i)+2$ boundary components 
obtained from $D$ by removing the interiors of those small disks. 
We choose the orientation of $E_{(b,i)}$ so that it agrees with the orientation of $D$. 
We call $E_{(b,i)}$ the {\it $E$-surface} for $b$. 
For example, the $1$-increasing braid $\sigma_1^2 \sigma_2^{-1} $ has the $E$-surface 
$E_{(\sigma_1^2 \sigma_2^{-1} , 1)}$
homeomorphic to a $3$-holed sphere.

\subsubsection*{Subcone $C_{(b,i)}$}  
Let us consider the $2$-dimensional subcone 
$C_{(b,i)}$ of $H_2(M_b, \partial M_b; {\Bbb R})$ spanned by $[F_b]$ and $[E_{(b,i)}]$ 
(Figure~\ref{fig_levelset}): 
$$C_{(b,i)} = \{x[F_b]+ y[E_{(b,i)}]\ |\ x>0, \ y > 0\}.$$ 
Let $\overline{C_{(b,i)}}$ denote the closure of $C_{(b,i)}$. 
We write $(x,y)= x[F_b]+ y[E_{(b,i)}] $. 
We prove the following theorem in Section~\ref{section_positivemain}.

\begin{thm}
\label{thm_positivemain}
For a pseudo-Anosov, $i$-increasing braid $b$ with $u(b,i)= u$, 
let $\mathcal{C}$ be the fibered cone containing $[F_b] $. 
We have  the following.

\begin{enumerate}
\item[(1)]
$C_{(b,i)} \subset \mathcal{C}$. 

\item[(2)] 
The fiber $F_{(x,y)}$ for each primitive integral class $(x,y) \in C_{(b,i)} $ 
 has genus $0$. 

\item[(3)] 
The monodromy 
$ \phi_{(x,y)}: F_{(x,y)} \rightarrow F_{(x,y)}$  for each primitive integral class $(x,y) \in C_{(b,i)} $
is conjugate to 
$$(\omega_1 \psi)  \cdots (\omega_{u-1} \psi) (\omega_u \psi) \psi^{m-1}: F_{(x,y)} \rightarrow F_{(x,y)},$$ 
where $m \ge 1$ depends on $(x,y)$,   
$\psi$ is periodic and 
each $\omega_j$ is reducible. 
Moreover there are 
homeomorphisms $\widehat{\omega}_j: S_0 \rightarrow S_0$ 
  for $j = 1, \ldots,  u$ on a surface $S_0$ 
determined by $b$ and 
an embedding $h: S_0 \hookrightarrow F_{(x,y)}$
such that the subsurface 
$h(S_0)$ of $F_{(x,y)}$ is the support of each $w_j$ 
and 
$$w_j|_{h(S_0)} =  h \circ \widehat{\omega}_j \circ h^{-1}.$$ 
\end{enumerate}
\end{thm}

The conclusion of Theorem~\ref{thm_positivemain} holds for 
$i$-decreasing braids as well. 
We now claim that  Theorem~\ref{thm_positivemain} implies  Theorem~\ref{thm_main}.

\begin{proof}[Proof of Theorem~\ref{thm_main}]
Suppose that Theorem~\ref{thm_positivemain} holds. 
Let $b \in B_n$ be a pseudo-Anosov braid such that $\pi_b(i)= i$. 
We consider the braid $b \Delta^{2k} \in B_n$ for $k \ge 1$. 
The full twist $\Delta^2$ is an element in  the center $Z(B_n)$ and 
$\Delta^2 = \sigma_j P_j$ holds for each $1 \le j \le n-1$, 
where $P_j$ is positive. 
Such properties  imply that 
$b \Delta^{2k}$ is  positive  for $k$ large. 
We fix such large $k$. 
Since $\Gamma(b)= \Gamma(b \Delta^{2k})$ in $\mathrm{Mod}(D_n)$, 
the braid $b \Delta^{2k}$ is certainly pseudo-Anosov. 
Hence it is $i$-increasing by Lemma~\ref{lem_positive}. 
One can apply Theorem~\ref{thm_positivemain} for this braid, 
and obtains the subcone $C_{(b \Delta^{2k}, i)}$. 
Consider the $k$th power of the disk twist  about  the disk $D_A$ 
bounded by the longitude of $\mathcal{N}(A)$: 
$$T_{D_A}^{k}: \mathcal{E}(A) \rightarrow  \mathcal{E}(A).$$ 
Since 
$A \cup T_{D_A}^k(\mathrm{cl}(b)) = A \cup \mathrm{cl}(b \Delta^{2k}) = \mathrm{br}(b \Delta^{2k})$, 
we have 
$S^3 \setminus \mathrm{br}(b) \simeq S^3 \setminus \mathrm{br}(b\Delta^{2 k})$.  
Let us set 
$$f_k:= h_{D_A, k}: M_b \rightarrow M_{b\Delta^{2 k}},$$ 
where $h_{D_A, k}$ is the homeomorphism in  (\ref{equation_homeo}). 
The isomorphism 
$${f_k}_*: H_2( M_b, \partial M_b) \rightarrow H_2(M_{b\Delta^{2 k}}, \partial M_{b\Delta^{2 k}})$$
sends $[F_b]$ to $[F_{b \Delta^{2 k}}]$. 
(Here we note that 
the above $k$ is suppose to be  large, 
but the homeomorphism $f_k$ 
makes sense for all integer $k$.)  
The pullback of the subcone $C_{(b \Delta^{2k}, i)} $ into $H_2(M_{b}, \partial M_{b})$ 
is a desired subcone contained in $\mathcal{C}$. 
\end{proof}

\begin{rem}
\label{rem_periodic}
If $F_{(x,y)}$ is a $(d+1)$-holed sphere, then 
the periodic homeomorphism $\psi: F_{(x,y)} \rightarrow F_{(x,y)}$ in Theorem~\ref{thm_positivemain}  
 is determined by the periodic braid $\rho= \sigma_1 \sigma_2 \ldots \sigma_{d-2} \sigma_{d-1}^2 \in B_d$. 
See the proof of Theorem~\ref{thm_positivemain}(3) in Section~\ref{subsection_claim3}. 
\end{rem}

\section{Proof of Theorem~\ref{thm_positivemain}}
\label{section_positivemain}

\begin{center}
\begin{figure}
\includegraphics[width=3.2in]{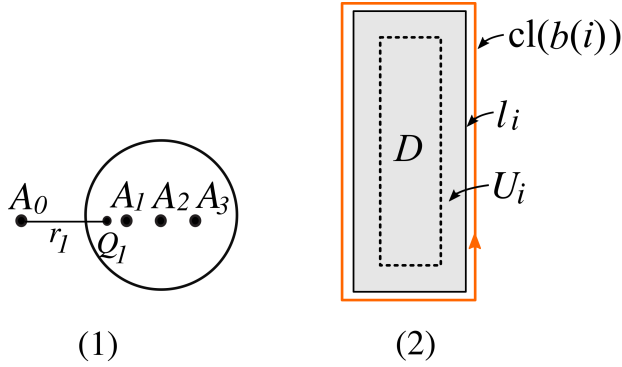} 
\caption{
(1) 
$A_0, \ldots, A_{n}$, $Q_i$, $r_i$ when $n=3$, $i=1$. 
(2) $ \partial D= \ell_i$ is a union of four segments. 
$U_i$ is an annulus in the figure.} 
\label{fig_standard}
\end{figure}
\end{center}

We fix integers $n \ge 3$ and $1 \le i \le n$. 
Throughout  Section~\ref{section_positivemain}, we assume that 
$b \in B_n$ is pseudo-Anosov and $i$-increasing with $u(b,i)= u$. 
We now choose an associated disk about the pair $(b,i)$ suitably. 
Let ${\Bbb D}$ denote the unit disk with the center $(0,0)$ in the plane ${\Bbb R}^2$. 
Let $J=  (-1,1) \times \{0\} \subset {\Bbb D}$ be the interval and let 
$A_0= (-2,0) $ be a point in ${\Bbb R}^2$.  
We denote by ${\Bbb D}_n$, the disk ${\Bbb D}$ with equally spaced $n$ points in $J$. 
Let us denote these $n$ points by $A_1, \ldots, A_n$ from left to right. 
We take a point $Q_i \ne A_i \in J$ between $A_{i-1}$ and $A_i$ 
so that the Euclidean distance $d(Q_i, A_i)$ is sufficiently small  (e.g. $d(Q_i, A_i) < \frac{1}{n+1}$). 
Let $r_i$ denote  the closed interval in $ [-2,1] \times \{0\}$ 
with endpoints  $A_0$ and $Q_i$. (Figure~\ref{fig_standard}(1).) 
We regard $b $ 
as a braid contained in the cylinder ${\Bbb D}^2 \times [0,1] \subset {\Bbb R}^3$  
and $b$ is based at $n$ points $A_1 \times \{0\}, \ldots, A_n \times \{0\}$. 
Since $\pi_b(i)= i$,  
one can take a representative of $b$ such that 
$b(i)$ is an interval in the cylinder: 
\begin{enumerate}
\item[$\diamondsuit 1$.] 
$b(i)= \displaystyle\bigcup_{0 \le t \le 1} A_i \times \{t\}$.
\end{enumerate}
Furthermore we may assume that  
$\partial D (= \ell_i)$ of an associated disk $D$ of  $(b,i)$
 is a union of the following four segments as a set (Figure~\ref{fig_standard}): 
 \begin{enumerate}
\item[$\diamondsuit 2$.] 
$\bigl(\displaystyle \bigcup_{-1 \le t \le 2} A_0 \times \{t\}\bigr) \cup \bigl( r_i \times \{-1\}\bigr)
 \cup \bigl( \displaystyle\bigcup_{-1 \le t \le 2} Q_i \times \{t\}\bigr) \cup \bigl( r_i \times \{2\}\bigr)$.
\end{enumerate}
Preserving $\diamondsuit 1,2$ 
we may further assume the following  (Figures~\ref{fig_standard}(2), \ref{fig_local}(1)): 
\begin{enumerate}
\item[$\diamondsuit 3$.] 
For a regular neighborhood $U_i$ of $\ell_i$ in $D$, 
we have $I(b,i) \subset U_i$.
\end{enumerate}
This is because 
every point $x \in D \cap K'$, 
where $K'$ is a component of $\mathrm{cl}(b- b(i))$, 
one can  slide $x$ along $K'$ so that 
the resulting point on $K'$ is in $U_i$. 
Said differently, 
preserving  $\partial D$ pointwise, 
we can modify a small neighborhood of $D$ near $K'$  
so that the resulting associated disk satisfies $\diamondsuit 3$. 

Under the conditions $\diamondsuit 1,2,3$  we have the following. 
For each $x \in D \cap K' \subset U_i$, 
there is a segment $s' \subset K'$ 
through  $x$  
such that $s'$ passes over  $b(i)$ 
since $b$ is $i$-increasing. 
See Figure~\ref{fig_local}(1). 
Such a local picture of $\mathrm{cl}(b)$ is used in the the next section. 
 Hereafter we  assume that  associated disks possess conditions 
 $\diamondsuit 1,2,3$.

\subsection{Proof of Theorem~\ref{thm_positivemain}(1)}
\label{subsection_claim1}

Let $s$ be the open  segment in $H_2( M_b, \partial M_b; {\Bbb R})$ 
with the endpoints 
$ \tfrac{n-1}{u}[E_{(b,i)}] = (0, \tfrac{n-1}{u})$ and $[F_b] = (1,0)$: 
\begin{equation}
\label{equation_s}
s = \{(x,y) \in C_{(b,i)} \ |\ y = -\dfrac{n-1}{u}x + \dfrac{n-1}{u}, \ 0 < x < 1\}.
\end{equation}
The ray of each point in $ C_{(b,i)}$ through the origin intersects with $s$.  
Thus for the proof of  (1), it suffices to prove that $s \subset \mathcal{C}$.

\begin{center}
\begin{figure}
\includegraphics[width=4in]{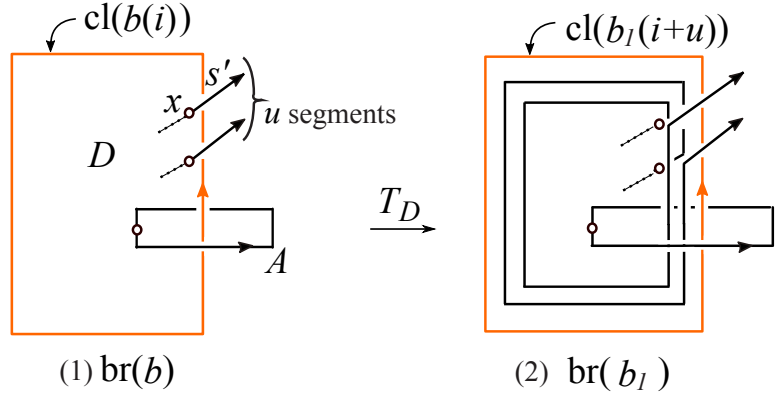} 
\caption{
Case: $b$ is $i$-increasing. 
(1) Associated disk $D$ with conditions $\diamondsuit$ 1,2,3. 
(2) $\mathrm{br}(b_1)$.
Circles $\circ$ indicate points of intersection  between $D$ and components of $\mathrm{br}(b-b(i))$. 
See  also Figure~\ref{fig_3chain}.} 
\label{fig_local}
\end{figure}
\end{center}

\begin{center}
\begin{figure}
\includegraphics[width=3.9in]{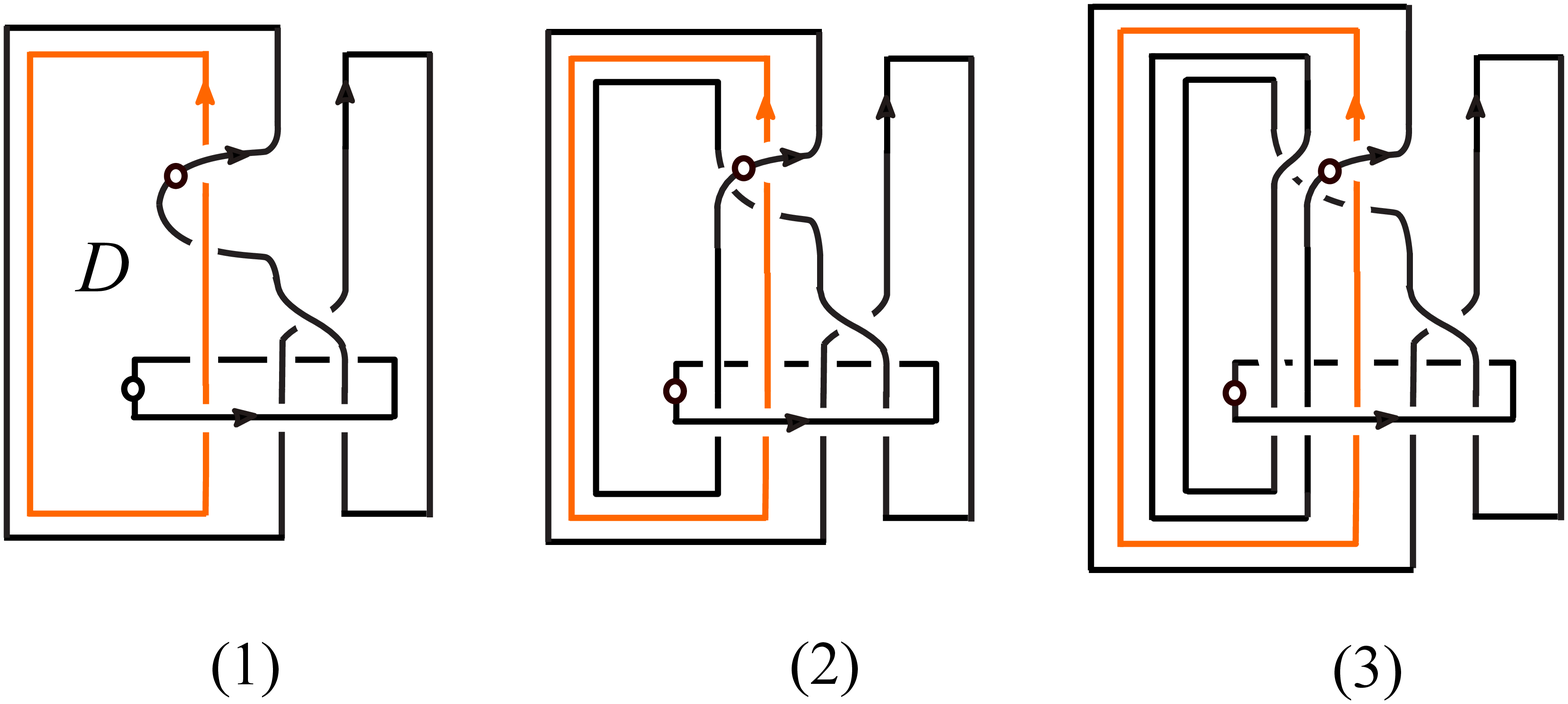} 
\caption{
Braided links for 
(1) $1$-increasing  $ \sigma_1^2 \sigma_2^{-1}$, 
(2) $2$-increasing  $(\sigma_1^2 \sigma_2^{-1})_1$ and 
(3) $3$-increasing  $(\sigma_1^2 \sigma_2^{-1})_2$.} 
\label{fig_3chain}
\end{figure}
\end{center}

We now introduce a sequence of braided links $\{\mathrm{br}(b_p)\}_{p=1}^{\infty}$ 
from an  $i$-increasing braid $b \in B_n$ 
such that  $M_{b_p} \simeq M_b$  for each $p \ge 1$. 
(We use the $1$-increasing braid 
$\sigma_1^2 \sigma_2^{-1} \in B_3 $ to illustrate the idea.) 
Let $D$ be an associated disk of the pair $(b,i)$. 
We take a disk twist 
$$T_D: \mathcal{E}(\mathrm{cl}(b(i))) \rightarrow \mathcal{E}(\mathrm{cl}(b(i)))$$  
so that the point of intersection $D \cap A$ becomes the center of the twisting about $D$, i.e. 
$T_D(D \cap A) = D \cap A$. 
We may assume that $T_D(A)= A$ as a set. 
Figure~\ref{fig_local} illustrates the image of the segment $s'$ under $T_D$. 
The condition $\diamondsuit 3$ ensures that 
$T_D$ equals the identity map outside a neighborhood of $U_i$ in $\mathcal{E}(\mathrm{cl}(b(i)))$. 
Then by $\diamondsuit 1,2$, it follows that 
$$T_D(\mathrm{br}(b - b(i))) \cup \mathrm{cl}(b(i))$$ 
is a braided link 
of some $(i+u)$-increasing braid with $(n+u)$ strands. 
We define $b_1 \in B_{n+u}$ to be such a braid. 
The trivial knot $T_D(A)(=A)$ becomes a braid axis of $b_1$. 
By definition of the disk twist, we have $M_{b_1} \simeq M_b$. 
See Figure~\ref{fig_3chain} for $\mathrm{br}((\sigma_1^2 \sigma_2^{-1})_1)$.

As discussed below, there is some ambiguity in defining  $b_1$. 
As we will see, the ambiguity is irrelevant 
for the study of pseudo-Anosov monodromies defined on fibers of fibrations on the mapping torus. 
Suppose that both $D$ and $D'$ are the associated disks of the pair $(b,i)$ 
with conditions $\diamondsuit 1,2,3$. 
We consider the disk twists $T_D$ and $T_{D'}$ with the above condition, 
i.e. both $D \cap A$ and $D' \cap A$ become  the center of the twisting about $D$ and $D'$ respectively.  
Observe that the resulting two links obtained from $D$ and $D'$  are equivalent: 
$$T_D(\mathrm{br}(b - b(i))) \cup \mathrm{cl}(b(i)) \sim T_{D'}(\mathrm{br}(b - b(i))) \cup \mathrm{cl}(b(i)).$$
They are  braided links,  
say $\mathrm{br}(b_1)$ and $\mathrm{br}(b'_1)$ 
of some braids $b_1, b_1' \in B_{n+u}$ respectively 
with the same axis $T_D(A)= A = T_{D'}(A)$. 
This means that a more stronger claim holds: 
$$(\mathrm{br}(b_1), A) \sim (\mathrm{br}(b'_1), A).$$
Thus $b_1$ and $b_1'$ are conjugate in $B_{n+u}$ by Theorem~\ref{thm_Morton}. 
In particular both $b_1$ and $b_1'$ are pseudo-Anosov 
(since the initial braid $b$ is pseudo-Anosov and $M_b$ is hyperbolic)  
and they have the same dilatation.  

To define $b_p$ for $p  \ge 1$, we consider the $p$th power 
$$T_D^{ p}: \mathcal{E}(\mathrm{cl}(b(i))) \rightarrow \mathcal{E}(\mathrm{cl}(b(i)))$$
using the above  $T_D$. 
As in the case of $p=1$, 
$$T_D^{p}(\mathrm{br}(b - b(i))) \cup \mathrm{cl}(b(i))$$ 
is a braided link of some $(i+ pu)$-increasing braid with $(n+pu)$ strands. 
We define $b_p \in B_{n+pu}$ to be such a braid. 
Then  $M_{b_p} \simeq M_b$. 
As in the case of $p=1$, such a braid $b_p$ is well-defined up to conjugate. 
We say that $b_p$ is {\it obtained from $b$ by the disk twist}. 
Clearly $u(b_p, i+pu) = u(b,i)$ for $p \ge 1$. 
See Figure~\ref{fig_3chain}.

Let us set  $$g_p:= h_{D,p}: M_b \rightarrow  M_{b_p},$$ 
where $h_{D,p}$ is the homeomorphism in (\ref{equation_homeo}). 
We consider the isomorphism 
$${g_p}_*: H_2(M_b, \partial  M_b) \rightarrow H_2(M_{b_p}, \partial M_{b_p}) .$$

\begin{lem}
\label{lem_iso_1}
For each integer $p \ge 1$, 
 ${g_p}_*$ sends $(0,1) \in \overline{C_{(b,i)}} $ to $(0,1) \in \overline{C_{(b_p, i+pu)}}$, 
and sends  
$(1,p) \in \overline{C_{(b,i)}}$ to $(1,0) \in \overline{C_{(b_p, i+pu)}}$. 
In particular for  integers $x, y \ge 1$ with $y = xp+r$ for $0 \le r < p$, 
${g_p}_*$ sends 
$(x,y) \in \overline{C_{(b,i)}}$ to $(x,r) \in \overline{C_{(b_p, i+pu)}}$. 
\end{lem}

\begin{proof}
We 
consider the {\it oriented sum} $F_{(x,y)}:= x F_b+ y E_{(b,i)}$. 
This is an oriented surface embedded  in $M_b$, 
and is obtained from the cut and past construction of 
parallel $x$ copies of $F_b$ and parallel $y$ copies of $E_{(b,i)}$. 
The orientation of $F_{(x,y)}$ agrees with those of $F_b$ and $E_{(b,i)}$. 
We have $[F_{(x,y)}] = (x,y) \in C_{(b,i)}$. 
Then $g_p$ sends $E_{(b,i)}$ to $E_{(b_p, i+pu)}$, 
and 
sends 
$F_{(1,p)}$ to $F_{b_p}$. 
Thus 
 ${g_p}_*$ sends $(0,1)  $ to $(0,1)$, and sends  $(1,p) $ to $(1,0) $. 
This completes the proof. 
\end{proof}

By the proof of Lemma~\ref{lem_iso_1}, 
$g_1$ sends 
$F_{(1,1)}=F_b+E_{(b,i)}$ to the fiber $F_{b_1}$ of a fibration on $M_b$ 
associated with $(1,1) \in C_{(b,i)}$. 
Since  the fibers $F_{(1,1)}$ and $F_b$ are  norm-minimizing, 
$E_{(b,i)} $ is also norm-minimizing.

\begin{proof}[Proof of Theorem~\ref{thm_positivemain}(1)]
We have $\|[F_b]\|= n-1$ and $\|[F_{b_p}]\| = n+ pu-1$ since $F_b$ and $F_{b_p}$ are fibers, 
and 
$\|[E_{(b,i)}]\|= u$ since $E_{(b,i)}$ is norm-minimizing. 
By Lemma~\ref{lem_iso_1}, 
$[F_{b_p}] =  (1,p) \in C_{(b,i)} $. 
Consider the rational class
$$c_p:= \dfrac{n-1}{n+pu-1}[F_{b_p}] = \Bigl(\dfrac{n-1}{n+pu-1},   \dfrac{p(n-1)}{n+pu-1}\Bigr).$$
Then $ \|c_p\|= n-1$ for $p \ge 1$. 
The ray of $[F_{b_p}]$ through the origin  is contained in some fibered cone for each $p \ge 1$. 
We easily check that $c_p$ lies on $s$ in (\ref{equation_s}). 
This means that three classes $[F_b]$, $c_p$ and $c_{p+1}$ with the same Thurston norm 
 are contained in $\mathcal{C}$. 
 Observe that the small segment $s'$ in $s$ connecting $[F_b]$ and $c_{p+1}$ contains $c_p$, 
 and  $s' \subset \mathcal{C}$ 
since  $\|\cdot \|$ is linear on each fibered cone. 
Moreover $c_p \to (0,  \tfrac{n-1}{u}) \in \partial s \subset \partial \mathcal{C}_{(b,i)}$ as $p \to \infty$. 
Putting all things together, 
we conclude that $s \subset \mathcal{C}$. 
This completes the proof. 
\end{proof}

\begin{rem}
\label{rem_e}
From the proof of Theorem~\ref{thm_positivemain}(1), 
one sees  the following: 
If $[E_{(b,i)}] \in \overline{C_{(b,i)}}$ is a fibered class, 
then 
$[E_{(b,i)}] \in \mathcal{C}$. 
Otherwise 
$[E_{(b,i)}] \in \partial{\mathcal{C}}$. 
See Figure~\ref{fig_levelset}(2)(3). 
\end{rem}

\subsection{Proof of Theorem~\ref{thm_positivemain}(2)}
\label{subsection_claim2}

We start with a simple observation: 
$\Delta^{2} \in B_n$ is $j$-increasing 
for each $1 \le j \le n$, 
and $u(\Delta^{2}, j) = n-1$ holds. 
The following lemma is immediate. 

\begin{lem}
\label{lem_bdelta}
If $b \in B_n$ is $i$-increasing, 
then $b \Delta^{2} \in B_n$ is $i$-increasing with 
$u(b  \Delta^{2}, i) = u(b,i)+ n-1$. 
\end{lem}

We explain the idea of  Theorem~\ref{thm_positivemain}(2). 
Let $D$ be the associated disk of the pair $(b,i)$. 
We have two types of  the disk twist. 
One is 
$T_{D_A}^k: \mathcal{E}(A) \rightarrow \mathcal{E}(A) $ which appears  in the proof of Theorem~\ref{thm_main} 
in Section~\ref{section_imonotonic} 
and the other is  $T_D^p: \mathcal{E}(\mathrm{cl}(b(i))) \rightarrow  \mathcal{E}(\mathrm{cl}(b(i)))$. 
If $k$ and $p$ are positive, then we obtain the $i$-increasing $b\Delta^{2k}$ from the former type $T_{D_A}^k$, 
and another increasing braid $b_p$ from the latter type $T_D^p$. 
Since both resulting braids are increasing, 
we can further apply  two types of the disk twist for the resulting braid. 
This is a key of the proof. 
Choosing two types of the disk twist alternatively, 
we get a sequence of increasing and pseudo-Anosov braids 
(since the initial braid $b$ is pseudo-Anosov).  
We shall see that the desired monodromies 
associated with primitive classes in $ C_{(b,i)}$ 
are given by these braids. 
\medskip

Let $p_1, \ldots, p_j$ be integers such that 
$p_1 \ge 0$ and $p_2, \ldots, p_j \ge 1$. 
Given an $i$-increasing braid $b \in B_n$ with $u(b,i)= u$, 
we define an integer $i[p_1, \ldots, p_j] \ge 1$ and an 
$i[p_1, \ldots,  p_j]$-increasing braid $b[p_1, \ldots, p_j]$ inductively as follows. 
\begin{itemize}
\item 
If $j=1$ and $p_1=0$, then $i[0]=i$ and $b[0]= b$. 
If $j=1$ and $p_1= p \ge 1$, then $i[p]= i+ pu$ and 
$b[p] = b_{p}$. 

\item 
If $j>1$ is even, then 
\begin{eqnarray*}
i[p_1, \ldots, p_{j-1}, p_j] &=& i[p_1, \ldots, p_{j-1}], 
\\
b[p_1, \ldots, p_{j-1}, p_j] &=& \bigl(b[p_1, \ldots, p_{j-1}]\bigr) \Delta^{2  p_j}. 
\end{eqnarray*}
The right-hand side is $ i[p_1, \ldots, p_{j-1}]$-increasing by  Lemma~\ref{lem_bdelta}.

\item
If $j>1$ is odd, then 
\begin{eqnarray*}
i[p_1, \ldots, p_{j-1}, p_j]  &=& i[p_1, \ldots, p_{j-1}]+ p_j u \bigl(b[p_1, \ldots, p_{j-1}], i[p_1, \ldots, p_{j-1}]\bigr), 
\\
b[p_1, \ldots, p_{j-1}, p_j] &=& \bigl(b[p_1, \ldots, p_{j-1}]\bigr)_{p_j}. 
\end{eqnarray*}
\end{itemize}
We say that $b[p_1, \ldots, p_j]$ has {\it length} $j$. 

\begin{ex} 
\label{ex_sequence}
\ 
\begin{enumerate}

\item[(1)] 
$b[p]= b_p$ by definition.

\item[(2)] 
Let $\beta= b \Delta^{2}$. 
Then $b[0,1]= \beta$ and 
$b[0,1,p]= \beta_p$.

\item[(3)]  
We have 
$b[0,p]=  b \Delta^{2  p}$ and 
$b[0,p,1]= (b\Delta^{2 p} )_1$, 
where $(b\Delta^{2 p} )_1$ is obtained from  $i$-increasing  $b\Delta^{2 p}$ 
by the disk twist. 
\end{enumerate}
\end{ex}

For each $k \ge 1$, 
let $f_k: M_b \rightarrow M_{b\Delta^{2  k}}$ 
be the homeomorphism which in the proof of Theorem~\ref{thm_main}. 
Consider the isomorphism 
${f_k}_*: H_2(M_b, \partial M_b) 
\rightarrow H_2(M_{b\Delta^{2  k}}, \partial M_{b\Delta^{2 k}}) $. 
We have the following property.

\begin{lem}
\label{lem_iso_2} 
For each integer $k \ge 1$, ${f_k}_*$ sends 
$(1,0) \in \overline{C_{(b,i)}}$ to $(1,0) \in \overline{C_{(b\Delta^{2  k},i)}}$, 
and sends 
$(k,1) \in \overline{C_{(b,i)}}$ to $(0,1) \in \overline{C_{(b\Delta^{2  k},i)}}$. 
In particular for integers $x,y \ge 1$ with $x= yk+r$ for $0 \le r <k$, 
then ${f_k}_*$ sends 
$(x,y) \in \overline{C_{(b,i)}}$ to $(r,y) \in \overline{C_{(b\Delta^{2  k},i)}}$. 
\end{lem}

\begin{proof}
The homeomorphism $f_k$ sends $F_b$ to $F_{b \Delta^{2 k}}$, 
and sends 
$F_{(k,1)}= kF_b + E_{(b,i)}$ to $E_{(b\Delta^{2  k}, i)}$. 
This implies that the claim holds. 
\end{proof}

\begin{proof}[Proof of Theorem~\ref{thm_positivemain}(2)]

Let $(x,y)   \in C_{(b,i)}$ be a primitive integral class. 
(Hence $x,y $ are positive integers with $\gcd(x,y)=1$.) 
We consider the continued fraction of  $y/x$ 
by the Euclidean algorithm 
$$\frac{y}{x}=
p_1+ \cfrac{1}{p_2 + 
           \cfrac{1}{p_3+ \cdots+
           \cfrac{1}{p_{j-1}+ 
           \cfrac{1}{p_{j}}}}}
:= p_1 + \frac{1}{p_2}{\genfrac{}{}{0pt}{}{}{+}}
          \frac{1}{p_3}{\genfrac{}{}{0pt}{}{}{+ \cdots +}}
           \frac{1}{p_{j-1}}{\genfrac{}{}{0pt}{}{}{+}}
           \frac{1}{p_{j}}{\genfrac{}{}{0pt}{}{}{}}
$$
with length $j$ and $p_j \ge 2$ and $p_1= 0$ if $0 < y< x$.  
There is another expression 
$$\frac{y}{x}= 
p_1 + \frac{1}{p_2}{\genfrac{}{}{0pt}{}{}{+}}
           \frac{1}{p_3}{\genfrac{}{}{0pt}{}{}{+ \cdots +}}
           \frac{1}{p_{j-1}}{\genfrac{}{}{0pt}{}{}{+}}
           \frac{1}{(p_{j}-1)}{\genfrac{}{}{0pt}{}{}{+}} 
             \frac{1}{1}{\genfrac{}{}{0pt}{}{}{}}     $$
with length $j+1$.
We choose one of the two expressions  
with odd length $\ell$: 
$$\frac{y}{x}= 
p_1 + \frac{1}{p_2}{\genfrac{}{}{0pt}{}{}{+}}
           \frac{1}{p_3}{\genfrac{}{}{0pt}{}{}{+ \cdots +}}
           \frac{1}{p_{\ell-1}}{\genfrac{}{}{0pt}{}{}{+}}
           \frac{1}{p_{\ell}}{\genfrac{}{}{0pt}{}{}{}}. $$
This encodes the fiber $F_{(x,y)}$ and its monodromy $\phi_{(x,y)}$. 
In fact 
Lemmas~\ref{lem_iso_1}, \ref{lem_iso_2} ensure that 
$$(g_{p_{\ell}}  f_{p_{\ell-1}} g_{p_{\ell-2}}   \cdots f_{p_2}  g_{p_1})_*:  H_2(M_b, \partial M_b) \rightarrow 
H_2(M_{b[p_1, \ldots, p_{\ell}]}, \partial M_{b[p_1, \ldots, p_{\ell}]})$$ 
sends 
$(x,y) = [x F_b+ y E_{(b,i)}]$ to $(1,0)$ which is the integral class of the $F$-surface of $b[p_1, \ldots, p_{\ell}]$. 
($g_{p_1} = id: M_b \rightarrow M_b$ if $p_1= 0$.) 
Thus $F_{(x,y)}$ has genus $0$. 
Moreover this means that 
one can take $F_{b[p_1, \ldots, p_{\ell}]}$ as a representative of $ (x,y) \in C_ {(b,i)}$ and 
the monodromy $\phi_{(x,y)}: F_{(x,y)} \rightarrow F_{(x,y)}$ is determined by 
$ b[p_1, \ldots, p_{\ell}]$. 
This completes  the proof. 
\end{proof}

We denote by $b_{(x,y)}$ the braid $b[p_1, \ldots, p_{\ell}]$ 
which determines $\phi_{(x,y)}$. 
Here is an example:    
If $(x,y)= (5,14)$, then 
$\frac{14}{5}= 
2 + \frac{1}{1}{\genfrac{}{}{0pt}{}{}{+}}
           \frac{1}{4}{\genfrac{}{}{0pt}{}{}{}}$ 
and  $\phi_{(5,14)}$ is determined by $b_{(5,14)}= b[2,1,4]$. 
If $(x,y)= (14,5)$, then 
$\frac{5}{14}=
0 + \frac{1}{2}{\genfrac{}{}{0pt}{}{}{+}}
\frac{1}{1}{\genfrac{}{}{0pt}{}{}{+}}
           \frac{1}{3}{\genfrac{}{}{0pt}{}{}{+}}
           \frac{1}{1}{\genfrac{}{}{0pt}{}{}{}}$      
and 
$\phi_{(14,5)}$ is determined by $b_{(14,5)}= b[0,2,1,3,1]$. 
%

\subsection{Proof of Theorem~\ref{thm_positivemain}(3)}
\label{subsection_claim3}

We begin with the following lemma.

\begin{lem}[Standard form]
\label{lem_normalform}
If $b \in B_n$ is $i$-increasing with $u(b,i)= u$, then 
$b$ is conjugate to an $n$-increasing braid $b'$  of the form 
$$b'= (w_1 \sigma_{n-1}^2)  \cdots (w_u \sigma_{n-1}^2),$$
where each $w_k$ is a word of $\sigma_1^{\pm 1}, \ldots, \sigma_ {n-2}^{\pm 1}$, 
but not $\sigma_{n-1}^{\pm 1}$, 
possibly $w_k = \emptyset$ for some $k$.
\end{lem}

Figure~\ref{fig_nincreasing}(1) shows the form of $b'$ in Lemma~\ref{lem_normalform} 
in case $u=2$.

\begin{proof}
We regard $b$ as a braid in  ${\Bbb D} \times [0,1]$. 
By $\diamondsuit 1$, $b(i)$ is an interval in ${\Bbb D} \times [0,1]$. 
If $i=n$, then 
$b$ is $n$-increasing and it  is not hard to see that 
a representative of $b$ is of the desired form in Lemma~\ref{lem_normalform}. 
Suppose that $b$ is $i$-increasing for $1 \le i < n$. 
We set $\sigma= \sigma_{n-1} \sigma_{n-2} \cdots \sigma_i$ 
if $1 \le i < n-1$ 
and $\sigma= \sigma_{n-1}$ if $i = n-1$. 
We consider the $n$-braid $b'= \sigma b \sigma^{-1}$ 
which is $n$-increasing with $u(b',n)= u$. 
We pull $b'(n)$ tight in ${\Bbb D} \times [0,1]$ 
and make it straight. 
Then a representative of $b'$ is of the desired form. 
\end{proof}

\begin{center}
\begin{figure}
\includegraphics[width=3in]{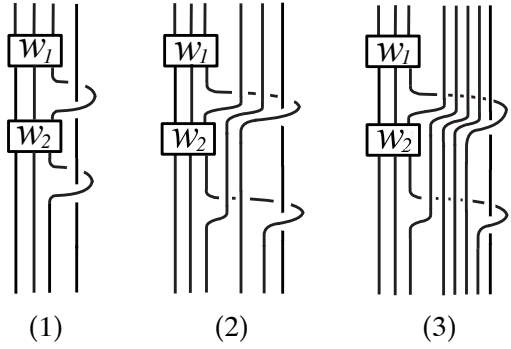} 
\caption{The figure illustrates 
how an initial braid $b$ generates  $\{b_p\}$. 
(1) $b=  w_1 \sigma_{3}^2 w_2 \sigma_{3}^2 = 
(\nu_1 \rho) (\nu_2 \rho) \in B_4$, 
where $\nu_j= w_j (\sigma_1 \sigma_2)^{-1}$.  
(2) $b_1 = (\nu_1 \rho) (\nu_2 \rho) \in B_6$.  
(3) $b_2 =(\nu_1 \rho) (\nu_2 \rho) \in B_8$.}
\label{fig_nincreasing}
\end{figure}
\end{center}

\begin{proof}[Proof of Theorem~\ref{thm_positivemain}(3)]
Since each $i$-increasing braid is conjugate to an $n$-increasing braid of a standard form 
in Lemma~\ref{lem_normalform}, 
we may assume that 
$b \in B_n$ is an $n$-increasing braid of the form 
$b= (w_1 \sigma_{n-1}^2) \cdots (w_u \sigma_{n-1}^2)$. 
Since $\rho \in B_n$ is the periodic braid such that 
$\rho= \sigma_1 \sigma_2 \cdots \sigma_{n-2} \sigma_{n-1}^2$ 
we have $\sigma_{n-1}^2 = (\sigma_1 \cdots \sigma_{n-2})^{-1} \rho$. 
Then $b$ is expressed as follows. 
$$b=  (\nu_1 \rho) \cdots (\nu_u \rho),$$
where $\nu_i= w_i (\sigma_1 \cdots \sigma_{n-2})^{-1}$ is 
written by a word of $\sigma_1^{\pm 1}, \cdots, \sigma_{n-2}^{\pm 1}$, 
but not $\sigma_{n-1}^{\pm 1}$. 
Each $\nu_j$ in $b$ is a reducible braid and $\rho$ in $b$ is the periodic braid. 
Let $\omega_j: F_b \rightarrow F_b$ denote a reducible representative whose mapping class is determined by  $\nu_j$, 
and let $\psi: F_b \rightarrow F_b$ denote a  periodic representative whose mapping class  determined by  $\rho$. 
The monodromy $\phi_b$ defined on $F_b$ is written by 
$\phi_b= (\omega_1 \psi) \cdots (\omega_u \psi) $.

Recall that ${\Bbb D}_{n-1}$ is the disk ${\Bbb D}$ with marked points $A_1, \cdots, A_{n-1}$. 
Let $S_0$ be an $n$-holed sphere obtained from ${\Bbb D}_{n-1}$ 
by removing the interiors of small $(n-1)$ disks with centers $A_1, \cdots, A_{n-1}$. 
Each $\nu_j$ as an $(n-1)$-braid determines a homeomorphism 
$\widehat{\omega}_j: S_0 \rightarrow S_0$. 
We may assume that 
$\widehat{\omega}_j$ fixes one of the boundary components 
corresponding to $\partial {\Bbb D}$  pointwise. 
It is clear that we have an embedding 
$h: S_0 \hookrightarrow F_b$ 
such that each $\omega_j$ in $\phi_b$ is reducible supported on the subsurface $h(S_0)$ 
and the restriction of $\omega_j$ to $h(S_0)$ is given by $h \circ \widehat{\omega}_j \circ h^{-1}$.

By the proof of Theorem~\ref{thm_positivemain}(2), 
$\phi_{(x,y)}: F_{(x,y)} \rightarrow F_ {(x,y)}$ 
associated with each primitive class $(x,y) \in C_{(b,i)}$ 
is determined by the braid of the form $b[p_1, \ldots, p_{\ell}]$. 
We now prove by the induction on  length $\ell$ that 
$$b[p_1, \ldots, p_{\ell}]=(\nu_1 \rho) \cdots ( \nu_{u-1} \rho) (\nu_u \rho) \rho^{m-1} 
=(\nu_1 \rho) \cdots ( \nu_{u-1} \rho) (\nu_u \rho^{m})$$ 
for some $m \ge 1$ depending on $(x,y)$. 
Here each $\nu_j$ in $b[p_1, \ldots, p_{\ell}]$ 
is a reducible braid which is an extension of $\nu_j$ in  $b$ 
and $\rho$ is the periodic braid with the degree of $b[p_1, \ldots, p_j]$. 
If this holds, then $\phi_{(x,y)}$ has a desired property as in Theorem~\ref{thm_positivemain}(3).  
Suppose that $\ell=1$. 
If $p_1=0$, then $b[0]= b $ and we are done. 
If $p_1 \ge 1$, then $b[p_1]= b_{p_1}$. 
Using the above expression of $b$ 
we observe that $b_{p_1}$ is written by 
$$b_{p_1}
=  (\nu_1 \rho) \cdots (\nu_u \rho) \in B_{n+ p_1 u}$$
(see Figure~\ref{fig_nincreasing}). 
We are done.

For $\ell \ge 2$, suppose that 
$b[p_1, \ldots, p_{\ell-1}]= (\nu_1 \rho_d) \cdots  (\nu_{u-1} \rho_d) (\nu_u \rho_d^{m})$ 
for some $m$, 
where $d$ is the degree of $b[p_1, \ldots, p_{\ell-1}]$. 
Consider  $b[p_1, \ldots, p_{\ell}]$ with length $\ell$. 
If $\ell$  is even, then by induction hypothesis
$$b[p_1, \ldots, p_{\ell}] = \bigl(b[p_1, \ldots, p_{\ell-1}]\bigr)  \Delta_d^{2 p_{\ell}}
= (\nu_1 \rho_d) \cdots  (\nu_{u-1} \rho_d) (\nu_u \rho_d^{m}) \Delta_d^{2 p_{\ell}}.$$ 
Since $\Delta_d^2= \rho_d^{d-1}$ 
we have $(\nu_u \rho_d^{m}) \Delta_d^{2 p_{\ell}} = \nu_u \rho_d^{m+ p_{\ell}(d-1)}$. 
Thus $b[p_1, \ldots, p_{\ell}]$ has a desired expression and we are done. 
If $\ell$ is odd, then by induction hypothesis again 
$$b[p_1, \ldots, p_{\ell}] = \bigr(b[p_1, \ldots, p_{\ell-1}]\bigl)_{p_\ell} = \bigl( (\nu_1 \rho_d) \cdots  (\nu_{u-1} \rho_d) (\nu_u \rho_d^{m})\bigr)_{p_\ell}.$$
As in the case of $\ell=1$, the braid in the right-hand side is expressed as 
$$ \bigl( (\nu_1 \rho_d) \cdots  (\nu_{u-1} \rho_d) (\nu_u \rho_d^{m})\bigr)_{p_\ell}
=   (\nu_1 \rho_{\dag}) \cdots  (\nu_{u-1} \rho_{\dag} )(\nu_u \rho_{\dag}^m), $$
where  $\dag$ is the degree of $b[p_1, \ldots, p_{\ell}] $. 
This completes the proof. 
\end{proof}

\section{Sequences of pseudo-Anosov braids with small normalized entropies}
\label{section_sequences}

In this section we prove Theorem~\ref{thm_nomalizedentropy}. 
We begin with an observation. 
Let $\Omega \subset \{a \in \mathcal{C} \ |\ \|a\| =1\}$ be a compact set in $H_2(M_b, \partial M_b; {\Bbb R})$  
and let $\mathcal{C}_{\Omega} \subset \mathcal{C}$ denote the cone over $\Omega$ through the origin. 
By Theroem~\ref{thm_Fried_1}(2) 
there is a constant $P= P(\Omega)>0$ depending on $\Omega$ 
such that 
$\mathrm{Ent}(a)< P$ for any $a \in C_{\Omega}$.  
This observation provides us many sequences of  pseudo-Anosov braids with small normalized entropies
from a single pseudo-Anosov braid $b$.

\begin{thm}
\label{thm_compact}
Suppose that $b$ is  a pseudo-Anosov braid whose  permutation has a fixed point. 
We fix any  $0 < \ell <\infty$. 
Let $\{(x_p, y_p)\}$ be a sequence of primitive integral classes in $C_{(b,i)}$ 
such that $y_p/x_p < \ell $ and $\|(x_p, y_p)\| \asymp p$. 
Then the sequence of pseudo-Anosov braids 
$\{b_{(x_p, y_p)}\}$ has a small normalized entropy. 
\end{thm}

\begin{proof}
If $\{(x_p, y_p)\}$ is the sequence under the assumption, 
then we have 
$d(b_{(x_p, y_p)}) \asymp \|(x_p, y_p)\| \asymp p$. 
Since $(1,0) \in C_{(b,i)} \subset \mathcal{C}$ 
and the slope of $y_p/x_p$ is bounded by $\ell$ from above,  
the set of projective classes 
$(x_p, y_p)$ is contained in some compact set in 
$\{a \in \mathcal{C} \ |\ \|a\| =1\}$ (Figure~\ref{fig_levelset}). 
Thus there is a constant $P= P(\ell)>1$ such that 
$\mathrm{Ent}(b_{(x_p, y_p)}) < P$ for any $p$.  
This completes the proof. 
\end{proof}

Let us discuss three sequences coming from Example~\ref{ex_sequence}. 
They are $\{b_p\}$, $\{\beta_p\}$ and $\{(b\Delta^{2  p} )_1\}$ varying $p$. 
It is not hard to see  that 
$d(b_p)$, $d(\beta_p)$, $d((b\Delta^{2  p} )_1) \asymp p$.

\begin{thm}
\label{thm_sequence}
For an $i$-increasing and pseudo-Anosov $b \in B_n$,  
we have the following on the sequences of pseudo-Anosov braids. 
\begin{enumerate}
\item[(1)] 
$\{b_p\}$ has a small normalized entropy 
if and only if $[E_{(b,i)}]$ is a fibered class. 

\item[(2)] 
For $\beta = b \Delta^2 \in B_n$, 
$\{\beta_p\}$  has a small normalized entropy and 
$\mathrm{Ent}(\beta_p) \to \mathrm{Ent}((1,1))$ as $p \to \infty$.

\item[(3)] 
$\{(b\Delta^{2 p} )_1\}$ has a small normalized entropy 
and $\mathrm{Ent}((b\Delta^{2 p} )_1) \to \mathrm{Ent}(b)$ 
as $p \to \infty$. 
\end{enumerate}
\end{thm}

\begin{proof}[Proof of Theorem~\ref{thm_sequence}]
For $a= (x,y) \in \overline{C_{(b,i)}}$, 
let $\underline{a} =\underline{(x,y)}$ denote its projective class. 
We have $\underline{[F_{b_p}]}= \underline{(1,p)} \to \underline{[E_{(b,i)}]}= \underline{(0,1)}$ 
as $p \to \infty$.  
If $[E_{(b,i)}]$ is a fibered class, then $[E_{(b,i)}] \in \mathcal{C}$ by Remark~\ref{rem_e} 
and $\mathrm{Ent}(b_p) \to \mathrm{Ent}([E_{(b,i)}])$ as $p \to \infty$ by Theorem~\ref{thm_Fried_1}(2). 
If $[E_{(b,i)}]$ is a non-fibered class, then 
$[E_{(b,i)}] \in \partial \mathcal{C}$ by Remark~\ref{rem_e}, 
and $\mathrm{Ent}(b_p) \to \infty$ as $p \to \infty$  by Theorem~\ref{thm_Fried_1}(3). 
We finish the proof of (1). 
We turn to (2). 
Since $[F_{\beta_p}]= (p+1,p)  \in C_{(b,i)}$, 
its projective class goes to $\underline{(1,1)}  $ as $p \to \infty$. 
Since $(1,1) \in C_{(b,i)} \subset \mathcal{C}$ by Theorem~\ref{thm_positivemain}(1), 
$\mathrm{Ent}(\beta_p) \to \mathrm{Ent}((1,1))$ as $p \to \infty$ by Theorem~\ref{thm_Fried_1}(2). 
This completes the proof of (2). 
Finally we prove (3). 
The fibered class of $F$-surface of $(b\Delta^{2  p} )_1$ 
is given by $(p+1,1) \in C_{(b,i)}$. 
Its projective class goes to $\underline{[F_b]}= \underline{(1,0)} $ as $p \to \infty$. 
Thus $\mathrm{Ent}((b\Delta^{2 p} )_1) \to \mathrm{Ent}(b)$ 
as $p \to \infty$. 
This completes the proof. 
\end{proof}

We use Theorem~\ref{thm_sequence}(1)(2) in Section~\ref{section_application}. 
For an application using (3), see \cite{HK17}.

\begin{proof}[Proof of Theorem~\ref{thm_nomalizedentropy}]
Suppose that  $b \in B_n$ is  pseudo-Anosov with $\pi_b(i)= i$. 
Let $\beta(k)$ denote $b \Delta^{2k} \in B_n$ for $k \ge 1$. 
Clearly  $\beta(k)$ is pseudo-Anosov with the same dilatation as $b$ (for any $k$) and 
$\beta(k)$ is  positive  for $k$ large. 
We fix such large $k$.  
By Lemma~\ref{lem_positive} $\beta(k)$ is $i$-increasing. 
If we let $z_p= (\beta(k) \Delta^{2p})_1$, 
then $M_{z_p} \simeq M_{\beta(k)} \simeq M_b$ holds for $p \ge 1$. 
By Theorem~\ref{thm_sequence}(3), 
$\{z_p\}$ has a small normalized entropy  and 
$\mathrm{Ent}(z_p) \to \mathrm{Ent}(\beta(k)) = \mathrm{Ent}(b)$ 
as $p \to \infty$. 
\end{proof}

Let $b_p^{\bullet}$ denote the braid obtained from $(i+ pu)$-increasing $b_p$ 
by removing the strand of the index  $i+ pu$. 
Taking its spherical element 
we have $S(b_p^{\bullet})$. 
A mild generalization of the sequence $\{b_p\}$ 
is the ones $\{b_p^{\bullet}\}$ and $\{S(b_p^{\bullet})\}$ varing $p$. 
Although  $b_p^{\bullet}$, $S(b_p^{\bullet})$ may not be pseudo-Anosov, 
they are frequently pseudo-Anosov.  
To be more precise, we need to consider    
the number of prongs of singularities in the stable foliation $\mathcal{F}_{b_p}$ for $b_p$ 
as we explained in Section~\ref{subsection_frombraids}. 
This is the motivation of the study in Section~\ref{section_stablefoliation}

\section{Stable foliation for the monodromy} 
\label{section_stablefoliation}

\begin{center}
\begin{figure}
\includegraphics[width=2.5in]{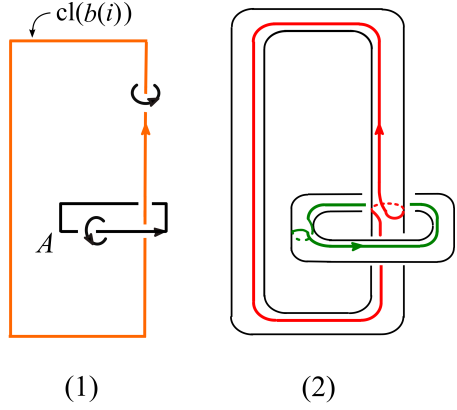} 
\caption{
Case: $b$ is $i$-increasing. 
(1)  Meridian and  longitude basis. 
(2) 
Two boundary slopes 
$\partial_{(b,A)} F_{(1,1)}$ (in green) on  $\mathcal{T}_{(b,A)}$ and 
$\partial_{(b,i)} F_{(1,1)} $ (in red) on $\mathcal{T}_{(b,i)}$ when $(x,y)= (1,1)$. 
}
\label{fig_slope}
\end{figure}
\end{center}

Let $b$ be pseudo-Anosov and  $i$-monotonic with the sign $\epsilon(b,i)= \epsilon$. 
For any primitive integral class $(x,y) \in C_{(b,i)}$, 
the oriented sum 
$F_{(x,y)} = x F_b+ y E_{(b,i)}$ is connected. 
Let $\mathcal{T}_{(b,A)}$ and $\mathcal{T}_{(b,i)}$ 
denote the tori $\partial \mathcal{N}(A)$ and $\partial \mathcal{N}(\mathrm{cl}(b(i)))$ respectively. 
Let us set 
$$\partial_{(b,A)} F_{(x,y)} = \partial F_{(x,y)} \cap \mathcal{T}_{(b,A)} \hspace{3mm}\mbox{and}\hspace{3mm}
\partial_{(b,i)} F_{(x,y)} = \partial F_{(x,y)} \cap \mathcal{T}_{(b,i)},$$ 
each of which is a single simple closed curve on the torus (since $\gcd(x,y)=1$). 
Recall that we chose the orientation of the axis for the $i$-monotonic $b$ in Section~\ref{section_imonotonic}. 
We use the meridian and  longitude basis $\{m_A, \ell_A\}$ for $\mathcal{T}_{(b,A)}$ 
to represent a homology class of a disjoint union of simple closed curves on $\mathcal{T}_{(b,A)}$. 
We also use the meridian and the longitude basis 
$\{m_i, \ell_i\}$ for  $\mathcal{T}_{(b,i)}$. 
Observe that the homology classes 
$[\partial_{(b,A)} F_{(x,y)}]$ and $[\partial_{(b,i)} F_{(x,y)}]$ are given by the pairs of integers 
\begin{equation}
\label{equation_slope}
[\partial_{(b,A)} F_{(x,y)}] = (- \epsilon y, x) \hspace{3mm}\mbox{and}\hspace{3mm} 
[\partial_{(b,i)} F_{(x,y)}]= (-\epsilon x, y).
\end{equation}
They are called  {\it boundary slopes} of $F_{(x,y)}$. 
See Figure~\ref{fig_slope}.

Let $ \phi_b: F_b \rightarrow F_b$ be the pseudo-Anosov monodromy of a fiber $F_b$ 
of the fibration on $M_b \rightarrow S^1$. 
The stable foliation $ \mathcal{F}_b$ of $\phi_b$ 
has singularities on each boundary component of $F_b$. 
Now we consider the suspension flow $\phi_b^t$ ($t \in {\Bbb R}$) on the mapping torus $M_b$. 
We obtain a disjoint union of simple closed curves $c_A = c_{(b,A)}$ on  $\mathcal{T}_{(b,A)}$ 
(possibly  a single simple closed curve) 
which is a  union of closed orbits for singularities in $\partial_{(b,A)} F_b$ under the flow. 
Similarly we have a disjoint union of simple closed curves $c_i = c_{(b,i)}$ on  $\mathcal{T}_{(b,i)}$ 
(possibly  a single simple closed curve again)  
which is a union of closed orbits for singularities in $\partial_{(b,i)} F_b$. 
(Figure~\ref{fig_sakuma3braid} depicts these closed curves for some pseudo-Anosov $3$-braid.) 
A useful tool is  {\it train track maps}  
which encode those data 
$\phi_b$, $\mathcal{F}_b$. 
They also enable us to compute homology classes $[c_A]$ and $[c_i]$. 

The following lemma is a consequence of Theorem~\ref{thm_Fried}(2) by Fried. 

\begin{lem}
\label{lem_foliation}
Let $\phi_{(x,y)}: F_{(x,y)} \rightarrow F_{(x,y)}$ be the monodromy of a fibration on $M_b \rightarrow S^1$
associated with a primitive integral class $(x,y)\in C_{(b,i)}$. 
Then the stable foliation $\mathcal{F}_{(x,y)}$  for  $\phi_{(x,y)}$ is 
$\mathfrak{i}([c_A], [\partial_{(b,A)} F_{(x,y)}])$-pronged  
at  $\partial_{(b,A)} F_{(x,y)}$, 
and is $\mathfrak{i}([c_i], [\partial_{(b,i)} F_{(x,y)}])$-pronged 
at  $\partial_{(b,i)} F_{(x,y)}$, 
where $\mathfrak{i}(\cdot, \cdot)$ means the geometric intersection number 
between homology classes of closed curves. 
\end{lem}

\begin{rem}
\label{rem_homology}
Every closed orbit of the suspension flow $\phi_b^t$ on the mapping torus $M_b$ 
travels around $S^1$ direction at least once. 
This implies that 
 $[c_A]$ has a non-zero first coordinate of the meridian and longitude basis for   $\mathcal{T}_{(b,A)}$, 
i.e., we have  $[c_A] = (k,\ell) \in {\Bbb Z}^2$ with $k \ne 0$, 
since the meridian for $\mathcal{T}_{(b,A)}$ corresponds to the flow direction. 
Similarly, $[c_i]$ has a non-zero second coordinate of the  meridian and longitude basis for $\mathcal{T}_{(b,i)}$, 
that is we have $[c_i]= (k', \ell') \in {\Bbb Z}^2$ with $\ell' \ne 0$, 
since the longitude for $\mathcal{T}_{(b,i)}$ corresponds to the flow direction in this case. 
\end{rem}

Recall that given a braid $b \in B_n$, 
we denote by $S(b) \in SB_n$, the spherical $n$-braid with the same word as $b$. 
For an $i$-increasing braid $b$ of pseudo-Anosov type, 
consider the braid $(b\Delta^{2 p} )_1= b[0,p,1]$ in Example \ref{ex_sequence}(3). 
This is an $i[0,p,1]$-increasing braid. 
Then we have its spherical braid $S((b\Delta^{2 p} )_1)$. 
We now define other braids obtained from $(b\Delta^{2 p} )_1$. 
Let $(b\Delta^{2 p} )_1^{\bullet}$ denote the braid obtained from $(b\Delta^{2 p} )_1$ 
by removing the strand of the index $i[0,p,1]$. 
Let $S((b\Delta^{2 p} )_1)$ and 
$S((b\Delta^{2 p} )_1^{\bullet})$ be the  spherical braids 
corresponding to $(b\Delta^{2 p} )_1$ and $(b\Delta^{2 p} )_1^{\bullet}$ respectively.  
Then we have the following result.

\begin{lem}
\label{lem_appendix}
Suppose that $b$ is an $i$-increasing braid of pseudo-Anosov type. 
For $p$ large, the braid  $(b\Delta^{2 p} )_1^{\bullet}$ and the spherical braids $S((b\Delta^{2 p} )_1)$, $S((b\Delta^{2 p} )_1^{\bullet})$ are all pseudo-Anosov 
with the same dilatation as $(b\Delta^{2 p} )_1$. 
\end{lem}

Before proving Lemma \ref{lem_appendix}, we recall a formula of the geometric intersection number $\mathfrak{i}([c],[c'])$ 
between two homology classes of simple closed curves $c$, $c'$ on a torus. 
Let $(p,q)$ and $(p',q')$ be primitive elements of ${\Bbb Z}^2$ which represent  $[c]$ and $[c']$ 
respectively. 
Then  $$\mathfrak{i}([c], [c']) =|pq'-p'q|.$$ 

\begin{proof}[Proof of Lemma \ref{lem_appendix}]
The fibered class of $F$-surface of $(b\Delta^{2 p} )_1$ is $(p+1,1) \in C_{(b,i)}$. 
We have $[\partial_{(b,A)} F_{(p+1,1)}] = (-1, p+1)$ and 
$[\partial_{(b,i)} F_{(p+1,1)}] = (-(p+1),1)$, see (\ref{equation_slope}). 
By Remark \ref{rem_homology}, one can write 
$[c_A] = (k,\ell)$ with $k \ne 0$ and 
$[c_i] = (k', \ell')$ with $\ell' \ne 0$. 
Then $\mathfrak{i}([c_A], [\partial_{(b,A)} F_{(p+1,1)}]) = |k(p+1)+ \ell|$ and 
$\mathfrak{i}([c_i], [\partial_{(b,i)} F_{(p+1,1)}]) = |k'+ \ell'(p+1)|$. 
Since $k \ne 0$ and $\ell' \ne 0$, 
these intersection numbers are increasing with respective to $p$ and 
they are clearly greater than $1$ when $p$ is large. 
 Then Lemma \ref{lem_foliation} says that 
  when $p$ is large, 
  the stable foliation $\mathcal{F}_{(p+1,1)}$  for the monodromy  $\phi_{(p+1,1)}$ is 
not $1$-pronged at each component of  $\partial_{(b,A)} F_{(p+1,1)} \cup \partial_{(b,i)} F_{(p+1,1)}$. 
By the discussion in Section \ref{subsection_stablefoliations}, we are done. 
\end{proof}

\section{Properties of $F$-surfaces and $E$-surfaces}
\label{section_properties}

\begin{center}
\begin{figure}
\includegraphics[width=4.5in]{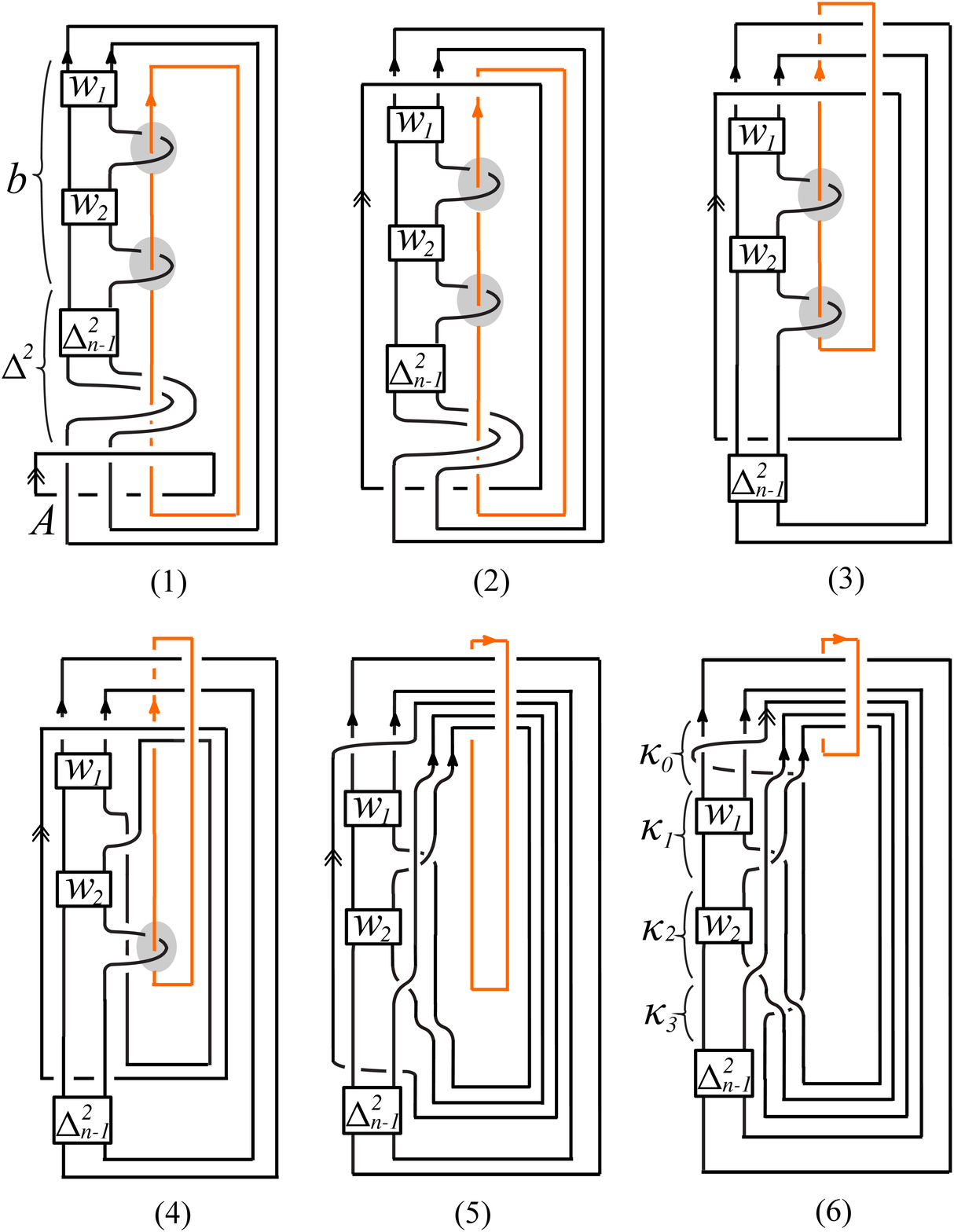} 
\caption{
Demonstration of Lemma~\ref{lem_ef} 
when $b$ is $n$-increasing with $u(b,n)=2$. 
(1) $\mathrm{br}(\beta)$ 
of $\beta=  w_1 \sigma_{n-1}^2 w_2 \sigma_{n-1}^2 \Delta^2$. 
(5)(6) $\mathrm{br}(\gamma)$ of  $\gamma=  \kappa_0\kappa_1 \kappa_2  \kappa_3 \Delta_{n-1}^2 $.}
\label{fig_fibercone}
\end{figure}
\end{center}

The aim of this section is to study  properties of  $E$-, $F$-surfaces 
and to present the technique used in the last section. 

\begin{lem}
\label{lem_ef}
For an $i$-increasing braid $b \in B_n$ 
with $u(b,i)= u$, 
we set $\beta= b \Delta^{2} \in B_n$.  
Then there is an $n$-increasing braid $\gamma \in B_{n+ u}$ 
such that 
$$(\mathrm{br}(\beta), \mathrm{cl}(\beta(i)), A_{\beta}) \sim (\mathrm{br}({\gamma}), A_{\gamma}, \mathrm{cl}(\gamma(n))).$$ 
In particular 
$M_b \simeq M_{\beta} \simeq M_{\gamma}$ and 
$E_{(\beta, i)}= F_{\gamma}$, $F_{\beta}= E_{(\gamma, n)}$ up to isotopy in $M_{\beta} $. 
Moreover 
if $b$ is pseudo-Anosov, then $\gamma$ is also pseudo-Anosov. 
\end{lem}

A similar claim holds for $i$-decreasing braids.

\begin{proof}
By Lemma~\ref{lem_normalform} we may assume that 
$b \in B_n$ is an $n$-increasing braid of a standard form 
$b= (w_1 \sigma_{n-1}^2) \cdots (w_u \sigma_{n-1}^2)$ 
containing $u $ subwords $\sigma_{n-1}^2$. 
Using the identity 
$$\Delta^2= \Delta_{n-1}^2 \sigma_{n-1} \cdots \sigma_2 \sigma_1 \sigma_1 \sigma_2  \cdots \sigma_{n-1} \in B_n,$$
we have (Figure~\ref{fig_fibercone}(1)) 
$$\mathrm{br}(\beta)= \mathrm{br}(b\Delta^2)= 
\mathrm{br}(w_1 \sigma_{n-1}^2 \cdots w_u \sigma_{n-1}^2  \Delta_{n-1}^2 \sigma_{n-1} \cdots \sigma_2 \sigma_1 \sigma_1 \sigma_2 \cdots \sigma_{n-1}).$$
We first deform  $\mathrm{br}(\beta)$ into a link as in Figure~\ref{fig_fibercone}(3). 
The same figure(1)(2)(3) tells us the process to get the desired link in (3). 
Then we perform the local moves in the shaded regions 
containing $u$ subwords $\sigma_{n-1}^2$ in $b$  
so that the link in question is  a union of the closure of some $n$-increasing braid 
$\gamma \in B_{n+u}$  and its braided axis, 
namely a braided link, 
see Figure~\ref{fig_fibercone}(3)(4)(5).  
As a result, 
$$(\mathrm{br}(\beta), \mathrm{cl}(\beta(n)), A_{\beta}) \sim  (\mathrm{br}({\gamma}), A_{\gamma}, \mathrm{cl}(\gamma(n))).$$
This expression says that $M_{\beta} \simeq M_{\gamma}$ and 
the $E$-, $F$-surfaces for $\beta$  are equal to 
the $F$-, $E$-surfaces for $\gamma$.  
Since $M_b \simeq M_{\beta}$ we are done. 
\end{proof}

Here we introduce a simple representative of 
$\gamma \in B_{n+u}$  in Lemma~\ref{lem_ef}. 
By the deformation as in (5)(6) of Figure~\ref{fig_fibercone}, 
we can take the following representative of $\gamma $.  
\begin{eqnarray*}
\gamma&=& \kappa_0 \kappa_1  \cdots \kappa_{u+1}  \Delta_{n-1}^2, \hspace{2mm} \mbox{where} 
\\
\kappa_0 &=& \sigma_{n-1} \sigma_{n-2} \cdots \sigma_1 \sigma_1 \sigma_2 \cdots \sigma_{n+u-1}, 
\\
\kappa_j&=& w_j \sigma_{n-1} \sigma_n \cdots \sigma_{n+u-j-1} \sigma_{n+u-j-2}^{-1} \cdots  \sigma_{n-1}^{-1} 
\hspace{0.5cm} \mbox{if}\ 1 \le j \le u-1,
\\
\kappa_u&=& w_u \sigma_{n-1}, 
\\
\kappa_{u+1}&=& \sigma_n^{-1} \hspace{3.5cm} \mbox{if}\  u=1, 
\\
\kappa_{u+1} &=& \sigma_{n+u-1}^{-1}  \sigma_{n+u-2}^{-1}\cdots \sigma_n^{-1} \hspace{0.5cm} \mbox{if}\ u \ge 2. 
\end{eqnarray*}
For example if $(n,u)= (3,2)$, then 
\begin{equation}
\label{equation_nu32}
\gamma= \kappa_0 \kappa_1 \kappa_2  \kappa_3 \Delta_2^2 
= \sigma_2 \sigma_1^2  \sigma_2 \sigma_3 \sigma_4
w_1 \sigma_2 \sigma_3 \sigma_2^{-1}  w_2 \sigma_2 \sigma_4^{-1} \sigma_3^{-1}  \sigma_1^2.
 \end{equation}
If $(n,u) = (3,3)$, 
then $\gamma=  \kappa_0 \kappa_1 \kappa_2 \kappa_3 \kappa_4 \Delta_2^2$, that is 
\begin{equation}
\label{equation_nu33}
\gamma = 
\sigma_2 \sigma_1^2  \sigma_2 \sigma_3 \sigma_4 \sigma_5 
w_1 \sigma_2 \sigma_3 \sigma_4 \sigma_3^{-1} \sigma_2^{-1} w_2 \sigma_2 \sigma_3 \sigma_2^{-1} 
w_3 \sigma_2 
\sigma_5^{-1}  \sigma_4^{-1} \sigma_3^{-1} \sigma_1^2. 
\end{equation}

Lemma~\ref{lem_ef} is used  in the following situation. 
Suppose that $\alpha \in B_{n+u}$ is a $j$-increasing  braid  and 
our task is to prove that $\alpha$ is pseudo-Anosov and its $E$-surface $E_{(\alpha, j)}$ is a fiber of a fibration on  $M_{\alpha} \rightarrow S^1$. 
(The conditions are needed to apply Theorem~\ref{thm_sequence}(1) for  $\alpha$.)
To do this, we need  to  find an $i$-increasing and pseudo-Anosov braid $b \in B_n$ with $u= u(b,i)$  
and need to check the resulting $n$-increasing braid $\gamma \in B_{n+u}$ in Lemma~\ref{lem_ef}  satisfies the property 
$$(\mathrm{br}(\gamma),  A_{\gamma}, \mathrm{cl}(\gamma(n))) \sim (\mathrm{br}(\alpha),  A_{\alpha}, \mathrm{cl}(\alpha(j))) ,$$ 
i.e. $\gamma$ is conjugate to $\alpha$  preserving the corresponding strand. 
If this equivalence holds, then by Lemma~\ref{lem_ef} together with the above equivalence $\sim$, our task is done. 
As a result $\{\alpha_p\}$ has a small normalized entropy  by  Theorem~\ref{thm_sequence}(1).

\section{Application} 
\label{section_application}

In the last section we prove Theorems~\ref{thm_pal}, \ref{thm_skp} and \ref{thm_spin}. 
We  first recall a study of pseudo-Anosov $3$-braids  \cite{Handel97,Matsuoka86}. 
Let $w$ be a word in $\sigma_1^{-1}$ and  $\sigma_2$. 
If both $\sigma_1^{-1}$ and $\sigma_2$ occur at least once in $w$, 
then we say that $w$ is a {\it pA word}. 
It is known that 
the $3$-braid  represented by a pA word is pseudo-Anosov. 
Conversely a $3$-braid $b$ is pseudo-Anosov, then 
there is a pA word $w$ such that 
the braid represented by $w$ is conjugate to $b$ up to a power of the full twist.   

 The stable foliation $\mathcal{F}_b$ is 
$1$-pronged at each boundary component of $F_b$ for each pseudo-Anosov $3$-braid $b$.  
Figure~\ref{fig_sakuma3braid}(3) exhibits a train track automaton. 
A train track map for the $3$-braid represented by a pA word $w$ 
is obtained from the closed loop  corresponding to $w$ in the automaton. 
For more details, see Ham-Song~\cite{HS07}.

\subsection{Palindromic/Skew-palindromic braids}

We define an anti-homomorphism 
\begin{eqnarray*}
rev: B_n &\rightarrow& B_n
\\
\sigma_{i_1}^{\mu_1} \sigma_{i_2}^{\mu_2} \cdots \sigma_{i_k}^{\mu_k} &\mapsto& 
\sigma_{i_k}^{\mu_k} \cdots \sigma_{i_2}^{\mu_2} \sigma_{i_1}^{\mu_1}, \hspace{3mm}\mu_j = \pm 1. 
\end{eqnarray*}
A braid $b \in B_n$ is   palindromic if $rev(b)= b$. 
Clearly $b \cdot rev(b)$ is  palindromic  for any $b \in B_n$. 
Let us consider another anti-homomorphism  
\begin{eqnarray*}
skew: B_n &\rightarrow& B_n
\\
\sigma_{i_1}^{\mu_1} \sigma_{i_2}^{\mu_2} \cdots \sigma_{i_k}^{\mu_k} &\mapsto&
\sigma_{n-i_k}^{\mu_k} \cdots \sigma_{n- i_2}^{\mu_2} \sigma_{n-i_1}^{\mu_1}, \hspace{3mm}\mu_j = \pm 1. 
\end{eqnarray*}
A braid $b \in B_n$ is skew-palindromic 
if $skew(b)= b$. 
Clearly $b \cdot skew(b)$ is  skew-palindromic  for any $b \in B_n$.

We now prove Theroems~\ref{thm_pal} and \ref{thm_skp} 
which indicate the asymptotic behaviors of minimal entropies among these subsets are quite distinct.

\begin{proof}[Proof of Theorem~\ref{thm_pal}] 
For  the surjective homomorphism  $\pi: B_n \rightarrow \mathcal{S}_n$ 
we write $\pi_j= \pi(\sigma_j) $. 
Suppose that an $n$-braid 
$b = \sigma_{i_1}^{\mu_1} \sigma_{i_2}^{\mu_2} \cdots \sigma_{i_k}^{\mu_k} $ is palindromic. 
Since $rev(b)= b$ we have 
$$(\pi_{rev(b)} =) \pi_{i_k} \cdots \pi_{i_2} \pi_{i_1} = \pi_{i_1} \pi_{i_2} \cdots \pi_{i_k} (= \pi_b). $$
Multiply the both side by $\pi_{i_1} \pi_{i_2} \cdots \pi_{i_k} $ from the left: 
$$(\pi_{i_1} \pi_{i_2} \cdots \pi_{i_k}) \cdot (\pi_{i_k} \cdots \pi_{i_2} \pi_{i_1}) 
=  (\pi_{i_1} \pi_{i_2} \cdots \pi_{i_k}) \cdot ( \pi_{i_1} \pi_{i_2} \cdots \pi_{i_k}) =\pi_b^2 .$$
Since $\pi_j^2= id$ 
the left-hand side equals  $id$. 
Hence $id= \pi_b^2$ which means that  the square $b^2$ is  pure. 
A theorem by Song~\cite{Song05} states that for a pseudo-Anosov pure element $b' \in B_n$, 
its dilatation has a uniform lower bound 
$2 + \sqrt{5} \le \lambda(b')$. 
In particular if $b'= b^2$, then 
$2 + \sqrt{5} \le \lambda(b^2) = (\lambda(b))^2$. 
This completes the proof. 
\end{proof}

\begin{center}
\begin{figure}
\includegraphics[width=4in]{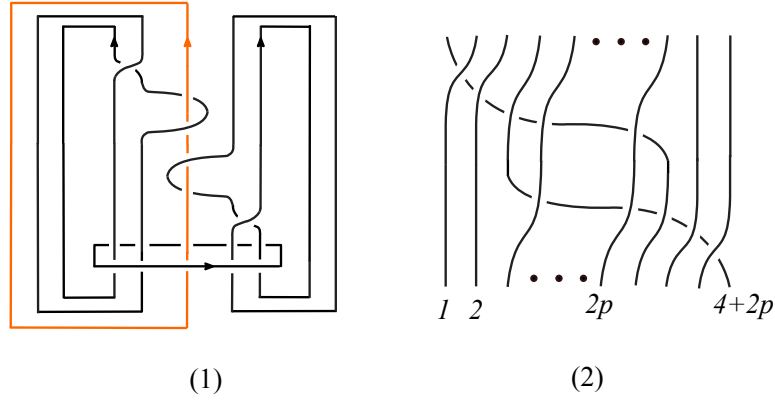} 
\caption{ 
(1) $\mathrm{br}(\xi)$. 
(2) Skew-palindromic  $\xi_p^{\bullet} \in B_{4+2p}$.}
\label{fig_122334}
\end{figure}
\end{center}

\begin{center}
\begin{figure}
\includegraphics[width=4.5in]{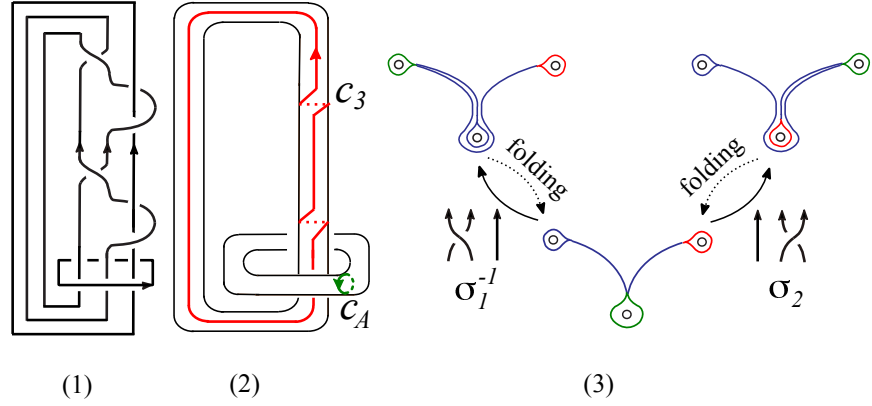} 
\caption{
(1) $\mathrm{br}(b)$ for $b= \sigma_1^{-1} \sigma_2^2 \sigma_1^{-1} \sigma_2^2$. 
(2)  $c_A \subset  \mathcal{T}_{(b,A)} $ and 
$c_3 \subset \mathcal{T}_{(b,3)}$. 
(3) Train track automaton.}
\label{fig_sakuma3braid}
\end{figure}
\end{center}

\begin{proof}[Proof of Theorem~\ref{thm_skp}]
We separate the proof into two cases, 
depending on the parity of the braid degree. 
We first prove 
$\log \delta(PA_{2n}) \asymp 1/n$. 
Let us take 
$\xi= \sigma_1 \sigma_2^2 \sigma_3^2 \sigma_4 \in B_5 $ (Figure~\ref{fig_122334}). 
The braid $\xi$ is  $3$-increasing with $u(\xi,3)=2$. 
We consider the disk twist about $D_{(\xi,3)}$. 
We obtain the braid $\xi_p$ which is $(3+2p)$-increasing for each $p \ge 1$. 
Observe that  $\xi_p^{\bullet}$ is a skew-palindromic braid with even degree  for each $p \ge 1$: 
$$\xi_p^{\bullet}
= (\sigma_1  \cdots \sigma_{1+ 2p} )(\sigma_3  \cdots \sigma_{3+ 2p}) \in B_{4+2p}.$$
(For the definition of $\xi_p^{\bullet}$, see Section~\ref{section_sequences}.) 
By the lower bound of  dilatations by Penner, 
 it is enough to prove that 
the sequence $\{\xi_p^{\bullet} \}$ has a small normalized entropy. 
We prove this in  the following two steps. 
In Step 1 
we prove that 
$\{\xi_p\}$ has a small normalized entropy. 
In Step 2 
we prove that 
the stable foliation $\mathcal{F}_{\xi_p}$  is not $1$-pronged at 
$\partial_{(\xi_p,3+2p)} F_{\xi_p}$ for $p \ge 1$. 
This tells us that 
$\xi_p^{\bullet} $ is pseudo-Anosov 
with the same dilatation as $\xi_p$. 
By Step 1 it follows that 
$\{\xi_p^{\bullet}\}$ has a small normalized entropy. 
\medskip

\noindent
{\bf Step 1.} 
The sequence  $\{\xi_p\}$ has a small normalized entropy. 
\medskip

\noindent
By Theorem~\ref{thm_sequence}(1) 
it suffices to prove that 
$\xi$ is pseudo-Anosov and $[E_{(\xi,3)}] $ is a fibered class. 
Consider a pseudo-Anosov braid $b = \sigma_1^{-1} \sigma_2^2 \sigma_1^{-1} \sigma_2^2 \in B_3$. 
It is  $3$-increasing with $u(b,3)=2$. 
For  $\beta = b \Delta^2$ we have $M_b \simeq M_{\beta}$.  
By Lemma~\ref{lem_ef} 
$(\mathrm{br}(\beta), \mathrm{cl}(\beta(3)), A_{\beta}) \sim (\mathrm{br}(\gamma), A_{\gamma}, \mathrm{cl}(\gamma(3)))$, 
where $\gamma \in B_5$ is the braid in (\ref{equation_nu32}) 
substituting $\sigma_1^{-1}$ for $w_1$ and $\sigma_1^{-1}$ for $w_2$. 
It is not hard to check that\footnote{There is a solution for the conjugacy problem on $B_n$ \cite{EM94}. 
The software {\it Braiding}~\cite{Gonzalez-Meneses} can be used to determine 
whether two braids are conjugate.}   
$\gamma$ is conjugate to $\xi $ in $B_5$ and 
 their  permutations  have  a common fixed point $3$.
Hence 
\begin{equation}
\label{equation_skew}
(\mathrm{br}(\beta), \mathrm{cl}(\beta(3)), A_{\beta}) \sim (\mathrm{br}(\xi), A_{\xi} , \mathrm{cl}(\xi(3))).
\end{equation}
In particular 
$E_{(\xi,3)} = F_{\beta}$ 
which means that $E_{(\xi,3)}$ is a fiber of a fibration on the hyperbolic mapping torus  $M_b \simeq  M_\xi$ over $S^1$.  
Thus $\xi$ is pseudo-Anosov. 
\medskip

\noindent
{\bf Step 2.} 
$\mathcal{F}_{\xi_p}$  is $(p+1)$-pronged at 
$\partial_{(\xi_p,3+2p)} F_{\xi_p}$ 
for $p \ge 1$. 
\medskip

\noindent
We read the singularity data of $\mathcal{F}_{\xi_p}$ from the monodromy 
$\phi_{\beta}: F_{\beta} \rightarrow F_{\beta}$ of the fibration on $M_{\beta} \rightarrow S^1$. 
First 
consider the suspension flow $\phi_b^t$ on the mapping torus $M_b$.  
Since $\mathcal{F}_b$   is 
$1$-pronged at each component of $F_b$, 
we have simple closed curves $c_A \subset \mathcal{T}_{(b,A)}$ and $c_3 \subset \mathcal{T}_{(b,3)}$ 
such that 
$[c_A] = (1,0)$, $[c_3] = (2,1) \in {\Bbb Z}^2$ (Figure~\ref{fig_sakuma3braid}(1)(2)). 

Next we turn to   $\beta = b \Delta^2 \in B_3$ and 
 the suspension flow $\phi_{\beta}^t$ on $M_{\beta} \simeq M_b$.  
We have simple closed curves $c_{(\beta,A)} \subset \mathcal{T}_{(\beta, A)}$ and 
$c_{(\beta, 3)} \subset \mathcal{T}_{(\beta, 3)}$. 
Since $\beta$ is the product of $b$ and  $\Delta^2$, 
we get $[c_{(\beta,A)}] =(1,0)+ (0,1) =(1,1)$. 
The first term $(1,0)$ comes from $[c_A] $ and the second one $(0,1)$ comes from $\Delta^2$. 
Similarly we have 
$[c_{(\beta, 3)}] = (2,1)+ (1,0)=(3,1)$. 
By (\ref{equation_skew}) we have 
$F_\beta = E_{(\xi,3)}$ and $E_{(\beta,3)} = F_\xi $. 
We also have 
$\mathcal{T}_{(\beta, A)} = \mathcal{T}_{(\xi, 3)}$ and 
$\mathcal{T}_{(\beta,3)} = \mathcal{T}_{(\xi,A)}$. 
Since 
$$p [F_{\beta}] + [E_{(\beta,3)}] = [F_\xi]+ p[E_{(\xi,3)}] =[ F_\xi+ p E_{(\xi,3)}]
= (1,p)  \in C_{(\xi,3)},$$ 
the stable foliation 
$\mathcal{F}_{(1,p)}$ associated with an integral class $(1,p) \in C_{(\xi,3)}$  is the stable foliation 
associated with $(p,1) \in C_{(\beta, 3)}$. 
By (\ref{equation_slope}) for $(x,y)=(p,1)$ 
$$[\partial_{(\beta,A)} (F_{\xi}+ p E_{(\xi,3)})]= (-1,p), \hspace{3mm} 
[\partial_{(\beta, 3)} (F_{\xi}+ p E_{(\xi,3)})] = (-p,1) \in {\Bbb Z}^2.$$
From 
$\mathfrak{i}([c_{(\beta,A)}], [\partial_{(\beta,A)}(F_{\xi}+ p E_{(\xi,3)})])= p+1$ and 
$\mathfrak{i}([c_{(\beta,3)}], [\partial_{(\beta,3)}(F_{\xi}+ p E_{(\xi,3)})])= p+3$ 
together with Lemma~\ref{lem_foliation}, 
one sees that 
$\mathcal{F}_{(1,p)}$ associated with  $(1,p) \in C_{(\xi,3)}$ 
is $(p+1)$-pronged at $\partial_{(\beta,A)}F_{(1,p)} (= \partial_{(\xi,3)} F_{(1,p)})$, 
and is $(p+3)$-pronged at $\partial_{(\beta,3)}F_{(1,p)} (= \partial_{(\xi,A)} F_{(1,p)})$.

Since $g_p: M_{\xi} \rightarrow M_{\xi_p}$ 
sends $F_{(1,p)}$ to $F_{\xi_p}$ 
the stable foliation $\mathcal{F}_{(1,p)}$ associated with $(1,p) \in C_{(\xi,3)}$ is identified with $\mathcal{F}_{\xi_p}$ via $g_p$.  
The boundary components $\partial_{(\xi,A)} F_{(1,p)}$ and $\partial_{(\xi,3)} F_{(1,p)}$ correspond to 
$\partial_{(\xi_p,A)} F_{\xi_p}$ and $\partial_{(\xi_p,3+2p)} F_{\xi_p}$ respectively 
via $g_p$. 
Thus $\mathcal{F}_{\xi_p}$  is $(p+1)$-pronged at 
$\partial_{(\xi_p,3+2p)} F_{\xi_p}$. 
This completes the proof of Step 2. 
\medskip

Next we prove  $\log \delta(PA_{2n+1}) \asymp 1/n$ following the above arguments in Steps 1,2. 
Take an initial braid 
$$\eta = \sigma_1 \sigma_2 \sigma_3 \sigma_4 \sigma_5 \sigma_3 \sigma_4 \sigma_3 \sigma_4  \sigma_5 \sigma_6 \sigma_7 \in B_8.$$ 
It is $4$-increasing with $u(\eta,4)=2$. 
Consider $\eta_p \in B_{8+ 2p}$ obtained from $\eta$ by the disk twist. 
Then $\eta_p^{\bullet}$ is a skew-palindromic braid with odd degree 
 for each $p \ge 1$: 
$$\eta_p^{\bullet} = (\sigma_1 \sigma_2 \cdots \sigma_{4+ 2p}) (\sigma_3 \sigma_4 \cdots \sigma_{6+2p}) \in B_{7+2p}.$$ 
For our purpose it suffices to prove that 
$\{\eta_p^{\bullet}\}$ has a small normalized entropy. 
Following Step 1 we first prove that $\eta$ is pseudo-Anosov and 
$[E_{(\eta, 4)}]$ is a fibered class. 
Consider a pseudo-Anosov braid $b = \sigma^{-1} \sigma_2^6 \Delta^2 \in B_3$ 
which is $3$-increasing with $u(b,3)= 5$. 
For  $\beta = b \Delta^2 $  
Lemma~\ref{lem_ef} tells us that 
$(\mathrm{br} (\beta), \mathrm{cl}(\beta(3)), A_{\beta}) \sim (\mathrm{br}(\gamma), A_{\gamma}, \mathrm{cl}(\gamma(3)))$, 
where $\gamma = \kappa_0 \kappa_1 \cdots \kappa_6 \Delta_2^2 \in B_8$. 
One sees that 
$\gamma$ is conjugate to $\eta$ in $B_8$. 
Since the permutation $\pi_{\eta}$ has a unique fixed point 
it follows that 
 $(\mathrm{br} (\beta), \mathrm{cl}(\beta(3)), A_{\beta}) \sim (\mathrm{br}(\eta), A_{\eta}, \mathrm{cl}(\eta(4)))$. 
 This expression says that 
 $E_{(\eta, 4)}= F_{\beta}$ is a fiber of a fibration on the hyperbolic  $M_b \simeq M_{\eta}$ over $S^1$. 
 Hence $\eta$ is pseudo-Anosov. 
 We conclude that $\{\eta_p\}$ has a small normalized entropy. 
 
 Following Step 2 one sees that 
 $\mathcal{F}_{\eta_p}$ is $(p+2)$-pronged at $\partial_{(\eta_p, 4+ 2p)}F_{\eta_p}$ for $p \ge 1$. 
 Thus $\eta_p^{\bullet}$ is  pseudo-Anosov with the same dilatation as $\eta_p$. 
 This completes the proof. 
\end{proof}

\subsection{Spin mapping class groups}

\begin{center}
\begin{figure}
\includegraphics[width=2.8in]{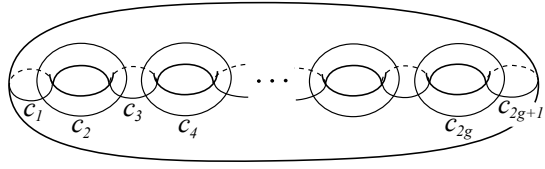} 
\caption{Simple closed curve $C_j$ on $\Sigma_g$.}
\label{fig_gen}
\end{figure}
\end{center}

In this section we prove Theorem~\ref{thm_spin}. 
We first recall a connection between $\mathcal{H}(\Sigma_g)$ and $\mathrm{Mod}(\Sigma_{0, 2g+2})$. 
Let $t_j \in \mathrm{Mod}(\Sigma_g)$ for $1 \le j \le 2g+1$ be the right-handed Dehn twist about the simple closed curve $C_j$ 
as in Figure~\ref{fig_gen}. 
Birman-Hilden~\cite{BH71} proved that 
$\mathcal{H}(\Sigma_g)$ is generated by $t_1, t_2, \ldots, t_{2g+1}$. 
In fact they prove that 
\begin{eqnarray*}
Q: \mathcal{H}(\Sigma_g) &\rightarrow& \mathrm{Mod}(\Sigma_{0, 2g+2})
\\
t_j &\mapsto & \mathfrak{t}_j
\end{eqnarray*}
sending $t_j$ to the right-handed half twist $\mathfrak{t}_j$ (see Section~\ref{subsection_frombraids}) 
is well-defined and it is a surjective homomorphism 
whose kernel is generated by the involution $\iota= [\mathcal{I}]$ as in Figure~\ref{fig_basis}. 
Using the relation between 
$\mathrm{Mod}(\Sigma_{0,2g+2}) $ and $SB_{2g+2}$ 
we have 
$$\mathcal{H}(\Sigma_g) / \langle \iota \rangle \simeq \mathrm{Mod}(\Sigma_{0,2g+2}) 
\simeq SB_{2g+2}/ \langle \Delta^2 \rangle.$$
It is well-known that 
$\phi \in \mathcal{H}(\Sigma_g)$ is pseudo-Anosov if and only if 
$Q(\phi)$ is pseudo-Anosov and in this case 
$\lambda(\phi) = \lambda(Q(\phi))$ holds. 
The following lemma is useful to find elements of the odd/even spin mapping class groups.

\begin{lem}[Theorem~6.1 in \cite{Hirose05} for (1), Theorem~3.1 in \cite{Hirose02} for (2)]
\label{lem_spin}
Suppose that  $g \ge 3$.  
\begin{enumerate}
\item[(1)] 
$ t_2,\  t_3,\  t_{j+1}t_j t_{j+1}^{-1},\  t_k^2 \in \mathrm{Mod}_g[\mathfrak{q}_1] $ 
for $4 \le j \le 2g$ and $1 \le k \le 2g+1$. 

\item[(2)] 
$t_{j+1}t_j t_{j+1}^{-1}, \ t_k^2 \in \mathrm{Mod}_g[\mathfrak{q}_0] $ 
for $1 \le j \le 2g$ and $1 \le k \le 2g+1$.
\end{enumerate}
\end{lem}
By the above result of Birman-Hilden, 
all mapping classes in Lemma~\ref{lem_spin} are  elements of $\mathcal{H} (\Sigma_g)$. 
Using the braid relations: 
$t_i t_j = t_j t_i$ if $|i - j| \ge 2$ and 
$ t_j t_{j+1} t_j = t_{j+1} t_j t_{j+1}$ for $1 \le j \le 2g$,  
we have 
$$t_j t_{j+1} t_j^{-1} = t_{j+1}^{-1} t_j t_{j+1} = t_{j+1}^{-2} ( t_{j+1}t_j t_{j+1}^{-1}) t_{j+1}^2.$$ 
Thus Lemma~\ref{lem_spin} tells us that 
$t_j t_{j+1} t_j^{-1} \in \mathrm{Mod}_g[\mathfrak{q}_1] $ 
for $4 \le j \le 2g$ and 
$t_j t_{j+1} t_j^{-1} \in \mathrm{Mod}_g[\mathfrak{q}_0] $ 
for $1 \le j \le 2g$.

\begin{center}
\begin{figure}
\includegraphics[width=5.5in]{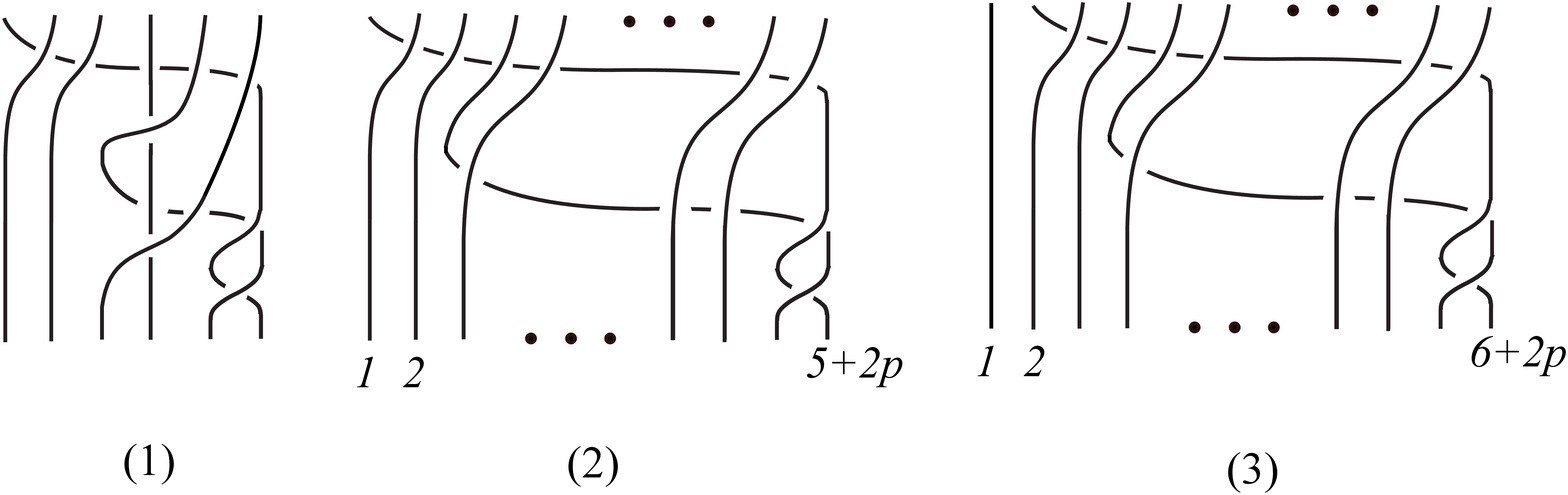} 
\caption{
(1) $o \in B_6$. 
(2) $o_p^{\bullet} \in B_{5+2p}$. 
(3) $sh(o_p^{\bullet}) \in B_{6+2p}$.}
\label{fig_oddspin}
\end{figure}
\end{center}

The following spin mapping classes  are used in 
the proof of Theorem~\ref{thm_spin}.

\begin{lem}
\label{lem_oddeven}
Let $p \ge 1$ be an integer. 
\begin{enumerate}
\item[(1)] 
$t_2 t_3 (t_4 t_5 \cdots t_{5+2p})^2 t_{5+2p} \in \mathrm{Mod}_g[\mathfrak{q}_1] $ 
for any $g \ge p+2$. 

\item[(2)] 
$(t_2 t_3 \cdots t_{5+2p})^2 t_{5+2p}^3 \in \mathrm{Mod}_g[\mathfrak{q}_0]$ 
for any $g \ge p+2$.  
\end{enumerate}
\end{lem}

\begin{proof}
We prove the lemma by the induction on $p$. 
We first prove (1). 
When $p =1$ 
$$t_2 t_3 (t_4 t_5 t_6 t_7)^2 t_7 
= t_2 \cdot t_3 \cdot t_4 t_5 t_4^{-1} \cdot t_4^2 \cdot t_6 t_7 t_6^{-1}
\cdot t_6 t_5 t_6^{-1} \cdot t_6^2 \cdot t_7^2  $$ 
which is an element of 
$\mathrm{Mod}_g[\mathfrak{q}_1] $ for $g \ge 3$ 
by Lemma~\ref{lem_spin}(1). 

Assume that $t_2 t_3 (t_4 t_5 \cdots t_{5+2(p-1)})^2 t_{5+2(p-1)} 
\in \mathrm{Mod}_g[\mathfrak{q}_1]$ for $g \geq p-1+2$. 
By the braid relations one verifies that 
\begin{eqnarray*}
&\ &  t_2 t_3 (t_4 t_5 \cdots t_{4+2(p-1)} t_{5+2(p-1)} t_{4+2p} t_{5+2p})^2  t_{5+2p}
\\
&=&  t_2 t_3 (t_4 t_5 \cdots t_{5+2(p-1)})^2 t_{5+2(p-1)} \cdot 
t_{5+2(p-1)}^{-2} \cdot t_{4+2p} t_{5+2p} t_{5+2(p-1)} t_{4+2p} \cdot t_{5+2p}^2. 
\end{eqnarray*}
Note that 
$t_j t_{j+1} t_{j-1} t_j= (t_j t_{j+1} t_j^{-1}) ( t_j t_{j-1} t_j^{-1})  t_j^2$. 
Then 
the assumption together with Lemma~\ref{lem_spin}(1)  implies that 
 $t_2 t_3 (t_4 t_5 \cdots t_{5+2p})^2 t_{5+2p} \in \mathrm{Mod}_g[\mathfrak{q}_1]$ for  $g \geq p+2$.  
 
 Let us turn to  (2). 
When $p=1$ 
$$(t_2 t_3 t_4 t_5 t_6 t_7)^2  t_7^3
= t_2 t_3 t_2^{-1} \cdot t_2^2 \cdot  t_4 t_3 t_4^{-1} \cdot t_4 t_5t_4^{-1} \cdot t_4^2 \cdot
t_6 t_7 t_6^{-1} \cdot t_6 t_5 t_6^{-1} \cdot t_6^2  \cdot  t_7^2
\cdot  t_7^2
$$
which is an element of $ \mathrm{Mod}_g[\mathfrak{q}_0]$ for $g \ge 3$. 

Assume that 
$(t_2 t_3 \cdots t_{5+2(p-1)})^2 t_{5+2(p-1)}^3 
\in \mathrm{Mod}_g[\mathfrak{q}_0]$ for any $g \geq p-1+2$. 
By the braid relations again, we have 
\begin{eqnarray*}
&\ & (t_2 t_3  \cdots t_{4+2(p-1)} t_{5+2(p-1)} t_{4+2p} t_{5+2p})^2 t_{5+2p}^3
\\
&=& (t_2 t_3 \cdots t_{5+2(p-1)})^2 t_{5+2(p-1)}^3 \cdot 
t_{5+2(p-1)}^{-4} \cdot t_{4+2p} t_{5+2p} t_{5+2(p-1)} t_{4+2p} \cdot t_{5+2p}^4. 
\end{eqnarray*}
By the assumption together with Lemma~\ref{lem_spin}(2) 
we have 
$(t_2 t_3  \cdots t_{5+2p})^2 t_{5+2p}^3 
\in \mathrm{Mod}_g[\mathfrak{q}_0]$ for $g \geq p+2$.  
This completes the proof. 
\end{proof}

The {\it shift map} $sh: B_n \rightarrow B_{n+1}$ 
is an injective homomorphism sending 
$\sigma_j$ to $\sigma_{j+1}$ for $1 \le j \le n-1$. 
Suppose that $b \in B_n$ is pseudo-Anosov. 
Then $S(sh(b)) \in SB_{n+1}$ is pseudo-Anosov  with the same dilatation as $b$ 
since 
$\widehat{\Gamma}(S(sh(b)) )$ is conjugate to $f_b=  \mathfrak{c}(\Gamma(b))$ in $\mathrm{Mod}(\Sigma_{0,n+1})$. 
(See Section~\ref{subsection_frombraids} for definitions $\Gamma$, $\widehat{\Gamma}$.)
We finally prove Theorem~\ref{thm_spin}. 

\begin{proof}[Proof of Theorem~\ref{thm_spin}(1)] 
Consider 
$o = \sigma_1 \sigma_2 \sigma_3 \sigma_4 \sigma_5 \sigma_3^2 \sigma_4 \sigma_5 \sigma_3 \sigma_5  \in B_6$ 
which is $4$-increasing with $u(o, 4)=2$ (Figure~\ref{fig_oddspin}). 
The braid $o_p $ is obtained from $o$ by disk twist for each $p \ge 1$. 
Then 
\begin{eqnarray*}
o_p^{\bullet} &=& 
\sigma_1 \sigma_2 (\sigma_3 \sigma_4 \cdots \sigma_{4+2p})^2 \sigma_{4+2p} \in B_{5+2p}, 
\\
S(sh(o_p^{\bullet})) &=&\sigma_2 \sigma_3 (\sigma_4 \sigma_5 \cdots \sigma_{5+2p})^2 \sigma_{5+2p} \in SB_{6+2p}.
\end{eqnarray*}
By Lemma~\ref{lem_oddeven}(1) 
$t_2 t_3 (t_4 t_5 \cdots t_{5+2p})^2 t_{5+2p} \in \mathrm{Mod}_{p+2}[\mathfrak{q}_1] $ 
for $p \ge 1$, and it is pseudo-Anosov if 
$S(sh(o_p^{\bullet}))$ is pseudo-Anosov. 
In this case they have the same dilatation. 
Thus by the relation between $o_p^{\bullet}$ and $S(sh(o_p^{\bullet}))$ 
it is enough to prove that $\{o_p^{\bullet}\}$  has a small normalized entropy. 
We first claim that $\{o_p\}$ has a small normalized entropy. 
By Theorem~\ref{thm_sequence}(1) 
it suffices to prove that  $o$ is a pseudo-Anosov and 
$[E_{(o,4)}]$ is a fibered class. 
Consider a $3$-braid 
$b=  \sigma_1^2 \sigma_2^2 \cdot  \sigma_2^2 \cdot  \sigma_2^2$ 
which is $3$-increasing with $u(b, 3)=3$. 
Let $\beta$ denote $ b \Delta^2$. 
By Lemma~\ref{lem_ef} 
$(\mathrm{br}(\beta), \mathrm{cl}(\beta(3)), A_{\beta}) \sim 
(\mathrm{br}(\gamma), A_{\gamma}, \mathrm{cl}(\gamma(3)))$, 
where $\gamma \in B_6$ is the braid in (\ref{equation_nu33}) 
substituting $\sigma_1^2$, $\emptyset$, $\emptyset$ for $w_1$, $w_2$, $w_3$ respectively. 
In this case  $\gamma $ is conjugate to $o$ in $B_6$. 
Since the permutation $\pi_o$ has a unique fixed point $4$, 
it follows that 
$(\mathrm{br}(\beta), \mathrm{cl}(\beta(3)), A_{\beta}) \sim (\mathrm{br}(o), A_{o}, \mathrm{cl}(o(4)))$.  
This tells us that 
$M_{\beta} \simeq M_o$ 
and 
$[E_{(o,4)}] = [F_{\beta}]$ is a fibered class. 
On the other hand 
$\beta$ is conjugate to $\sigma_1^4 \sigma_2^{-2} \Delta^4$ in $B_3$ which means that $\beta$ is pseudo-Anosov. 
Thus $M_{\beta} \simeq M_o$ is hyperbolic and $o$ is pseudo-Anosov. 

Next we prove that 
$o_p^{\bullet} $ is pseudo-Anosov with the same dilatation as $o_p$ for $p \ge 1$. 
By the same argument as in the proof of Theorem~\ref{thm_skp} 
one sees that $\mathcal{F}_{o_p}$  
is $(p+2)$-pronged at $\partial_{(o_p, 4+2p)} F_{o_p}$. 
Thus   $o_p^{\bullet} $ has the desired property for $p \ge 1$. 
We finish the proof of (1).  

We turn to (2). 
Let us consider  $v= (\sigma_1 \sigma_2 \sigma_3 \sigma_4 \sigma_5)^2 \sigma_1 \sigma_2 \sigma_5^3 \in B_6$ 
which is $3$-increasing with $u(v,3)= 2$. 
Let $v_p \in B_{6+ 2p}$ be the braid obtained from $v$ by the disk twist. 
Then $v_p$ is $(3+2p)$-increasing and 
\begin{eqnarray*}
v_p^{\bullet} &=& (\sigma_1 \sigma_2 \cdots \sigma_{4+2p})^2 \sigma_{4+2p}^3 \in B_{5+2p}, 
\\
S(sh(v_p^{\bullet}))&=& (\sigma_2 \sigma_3 \cdots \sigma_{5+2p})^2 \sigma_{5+2p}^3 \in SB_{6+2p}.
\end{eqnarray*}
By Lemma~\ref{lem_oddeven}(2) 
it is enough to prove that $\{v_p^{\bullet}\}$ has a small normalized entropy. 
To do this we first prove that 
$\{v_p\}$ has a small normalized entropy. 
Consider a pseudo-Anosov $3$-braid 
$$b = \sigma_1^2 \sigma_2^{-2} \Delta^4  = \sigma_1^3 \sigma_2^2 \sigma_1 \Delta^2 = 
\sigma_1^3 \sigma_2^2 \cdot \sigma_1^2 \sigma_2^2 \cdot  \sigma_1 \sigma_2^2$$ 
which is $3$-increasing with $u(b, 3)= 3$. 
Lemma~\ref{lem_ef} tells us that 
for $\beta = b \Delta^2$ we have 
$(\mathrm{br}(\beta), \mathrm{cl}(\beta(3)), A_{\beta}) \sim 
(\mathrm{br}(\gamma), A_ {\gamma}, \mathrm{cl}(\gamma(3)))$, 
where $\gamma \in B_6$ is the braid in (\ref{equation_nu33}) 
substituting $\sigma_1^3$ for $w_1$, $\sigma_1^2$ for $w_2$ and $\sigma_1$ for $w_3$. 
One sees that $\gamma$ is conjugate to $v $ in $B_6$. 
Thus 
$(\mathrm{br}(\beta), \mathrm{cl}(\beta(3)), A_{\beta}) \sim 
(\mathrm{br}(v), A_ {v}, \mathrm{cl}(v(3)))$. 
This implies that $[E_{(v,3)}]= [F_{\beta}]$ is a fibered class of the hyperbolic  $M_{\beta} \simeq M_v$, 
and hence $v$ is pseudo-Anosov. 
By Theorem~\ref{thm_sequence}(1), $\{v_p\}$ has a small normalized entropy.

One sees that  $\mathcal{F}_{v_p}$  is 
$(p+3)$-pronged at $\partial_{(v_p,3+2p)}F_{v_p}$. 
Thus $v_p^{\bullet} $ is pseudo-Anosov with the same dilatation as $v_p$ for $p \ge 1$. 
This  completes the proof. 
\end{proof}



\begin{thebibliography}{99}

\bibitem{AY81}
P. Arnoux and J-P. Yoccoz, 
Construction de diff\'eomorphismes pseudo-Anosov,
C. R. Acad. Sci. Paris S\'er. I Math. 292 
 (1981), no. 1, 75--78. 


\bibitem{BGP14}
A. J. ~Berrick, V. ~Gebhardt and L. ~Paris, 
Finite index subgroups of mapping class groups, 
Proc. London Math. Soc. (3) 108 (2014), no. 3,  575--599. 

\bibitem{BH71}
J.~Birman and H.~Hilden, 
On mapping class groups of closed surfaces as covering spaces, 
Advances in the theory of Riemann surfaces, 
Annals of Math Studies 66, Princeton University Press
(1971), 81--115. 


\bibitem{Dehornoy13}
P.~Dehornoy, 
Small dilatation homeomorphisms as monodromies of Lorenz knots, 
Institut Mittag-Leffler Preprints Series: IML Workshop on Growth and Mahler Measures in Geometry and Topology 
(2013), 1--9. 


%
\bibitem{Dye78}
R.~H.~Dye, 
On the Arf invariant, 
J. Algebra 53  (1978), no. 1, 36--39.

\bibitem{EM94}
E.~A.~Elrifai and H.~R.~Morton, 
Algorithms for positive braids, 
Quart. J. Math. Oxford Ser. 45 
(1994), no. 2, 479--497. 


%

\bibitem{FLM08}
B.~ Farb, C.~ J.~ Leininger and D.~ Margalit, 
The lower central series and pseudo-Anosov dilatations, 
American Journal of Mathematics  130 (2008), no. 3, 799--827. 

\bibitem{FLM11}
B.~Farb, C.~J.~ Leininger and D.~Margalit, 
Small dilatation pseudo-Anosov homeomorphisms and 3-manifolds, 
Adv. Math. 228 (2011), no. 3, 1466--1502. 


\bibitem{FM12}
B.~Farb and D.~Margalit.
A primer on mapping class groups, 
vol. 49, Princeton Mathematical Series, 
Princeton University Press, Princeton, NJ, 2012.


\bibitem{Fried79}
D.~Fried, 
Fibrations over $S^1$ with pseudo-Anosov monodromy, 
Expos\'e~14 in `Travaux de Thurston sur les surfaces' by A.~Fathi, F.~Laudenbach and V.~Poenaru,
Ast\'erisque, 66-67, 
Soci\'et\'e Math\'ematique de France, Paris (1979),  251--266.  

\bibitem{Fried82}
D.~Fried, 
Flow equivalence, hyperbolic systems and a new zeta function for flows, 
Comment. Math. Helv. 57 (1982), no. 2, 237--259. 


\bibitem{Gonzalez-Meneses}
J.~ Gonz\'alez-Meneses, 
Braiding is available at \verb#http://personal.us.es/meneses/software.php#


\bibitem{HS07}
J.~Y.~Ham and W.~T.~Song, 
The minimum dilatation of pseudo-Anosov $5$-braids, 
Experiment. Math. 16 (2007), no. 2, 167--179. 


\bibitem{Handel97}
M.~Handel, 
The forcing partial order on the three times punctured disk, 
Ergodic Theory Dynam. Systems 17 (1997), no. 3, 593--610. 

\bibitem{Hironaka14} 
E. Hironaka, 
Penner sequences and asymptotics of minimum dilatations for subfamilies of the mapping class group, 
Topology Proc. 44 (2014), 315--324. 


\bibitem{HK06} 
E.~Hironaka and E. Kin, 
A family of pseudo-{A}nosov braids with small dilatation
Algebr. Geom. Topol. 6 (2006), 699--738. 




\bibitem{Hirose02}
S.~Hirose, 
On diffeomorphisms over surfaces trivially embedded in the 4-sphere, 
Algebr. Geom. Topol. 2  (2002), 791--824.
%
\bibitem{Hirose05}
S. ~Hirose, 
Surfaces in the complex projective plane and their mapping class groups, 
Algebr. Geom. Topol. 5  (2005), 577--613. 

\bibitem{HK17}
S.~Hirose and E.~Kin, 
The asymptotic behavior of the minimal pseudo-Anosov dilatations in the hyperelliptic handlebody groups, 
Q. J. Math. 68 (2017),
no. 3, 1035--1069. 

\bibitem{Kin15}
E.~Kin, 
Dynamics of the monodromies of the fibrations on the magic $3$-manifold,  
New York J. Math.  21 (2015),  547-599. 

\bibitem{KM18} 
S.~Kojima and G.~McShane, 
Normalized entropy versus volume for pseudo-Anosovs, 
Geom. Topol. 2 (2018), no.  4, 2403--2426.


\bibitem{KT11}
E.~Kin and M.~Takasawa, 
Pseudo-Anosov braids with small entropy and the magic $3$-manifold, 
Comm. Anal. Geom. 19 (2011), no. 4, 705--758. 

\bibitem{Margalit18}
D.~Margalit, 
Problems, questions, and conjectures about mapping class groups, 
preprint, (2018). arXiv:1806.08773


\bibitem{Matsuoka86}
T.~Matsuoka, 
Braids of periodic points and a $2$-dimensional analogue of Sharkovskii's ordering, 
Dynamical Systems and Nonlinear Oscillations. Ed. G. Ikegami. World Scientific Press 
(1986) 58--72.


\bibitem{McMullen00} 
C.~T.~McMullen, 
 Polynomial invariants for fibered $3$-manifolds and Teichm\"uller geodesics for foliations,
Ann. Sci. \'Ecole Norm. Sup. (4) 33 
(2000), no. 4, (519--560) 
     
     

\bibitem{Morton78}
H.~R.~Morton, 
Infinitely many fibred knots having the same Alexander polynomial, 
Topology 17 (1978), 
no. 1, 101--104. 


\bibitem{Penner91}
R.~C.~Penner, 
Bounds on least dilatations, 
Proc. Amer. Math. Soc. 113 (1991),
no. 2, 443--450. 


\bibitem{Song05}
W.~T.~Song, 
Upper and lower bounds for the minimal positive entropy of pure braids. 
Bull. London Math. Soc. 37 (2005), no. 2, 224--229. 





\bibitem{Thurston86}
W.~P.~Thurston, 
A norm for the homology of $3$-manifolds, 
Mem. Amer. Math. Soc. 59 (1986), no. 339, 99--130. 


\bibitem{Thurston88}
W.~P.~Thurston, 
On the geometry and dynamics of diffeomorphisms of surfaces,
  Bull. Amer. Math. Soc. (N.S.) 19 (1988), no. 2, 417--431. 


 \bibitem{Thurston98} 
W.~P.~Thurston, 
Hyperbolic structures on 3-manifolds II: Surface groups and 
3-manifolds which fiber over the circle, preprint, 
arXiv:math/9801045


\bibitem{Tsai09} 
C.~Y.~Tsai, 
The asymptotic behavior of least pseudo-Anosov dilatations, 
Geom. Topol. 13 (2009),
no. 4, 2253--2278. 

\bibitem{Valdivia12}
A.~D.~Valdivia, 
Sequences of pseudo-Anosov mapping classes and their asymptotic behavior,
New York J. Math. 18
(2012), 609--620. 

\end{thebibliography}
\end{document}